\documentclass[12pt,a4paper,oneside,reqno]{amsproc}
\usepackage{amsmath} 
\usepackage{amssymb,amsfonts,mathrsfs}
\usepackage[colorlinks=true,linkcolor=blue,citecolor=blue]{hyperref}
\usepackage[driver=pdftex,lmargin=2.5cm,rmargin=2.5cm,heightrounded=true,centering]{geometry}
\usepackage{math50}
\usepackage{local}
\usepackage{epsfig}
\usepackage{graphicx}
\usepackage{xy}
\xyoption{all}
\begin{document}
\title[Existence results for the fractional Nirenberg problem]
{Existence results for the fractional Nirenberg problem}
\author{Yan-Hong Chen}
\address{School of Mathematical Science, Nankai University, Tianjin 300071, P.R. China}
\email{cyh1801@163.com}
\author{Chungen Liu}
\address{School of Mathematical Science and LPMC, Nankai University, Tianjin 300071, P.R. China}
\email{liucg@nankai.edu.cn}
\author{Youquan Zheng}
\address{School of Science, Tianjin University, Tianjin 300072, P. R. China.}
\email{zhengyq@tju.edu.cn}
\newcommand{\optional}[1]{\relax}
\setcounter{secnumdepth}{3}
\setcounter{section}{0} \setcounter{equation}{0}
\numberwithin{equation}{section}
\newcommand{\MLversion}{1.1}
\thanks{The third author was partially supported by NSFC of China (11271200)}
\date{\today}
\begin{abstract}
We consider the fractional Nirenberg problem on the standard sphere $\mathbb{S}^n$ with $n\geq 4$. Using
the theory of critical points at infinity, we establish an Euler-Hopf type formula and 
obtain some existence results for curvature satisfying assumptions of Bahri-Coron type.
\end{abstract}
\maketitle
\section{Introduction}
The famous Nirenberg problem in conformal geometry is: on the sphere $\mathbb{S}^n$ ($n\geq 2$) with standard metric $g_0$, is there a representation $g$ of the conformal class $[g_0]$ such that $g$ has scalar curvature (Gauss curvature for $n = 2$) equal to a prescribed function $K$? This problem is equivalent to the following equations
\begin{equation*}\label{e:Yamabe1}
-\Delta_{g_0}u + 1 = K e^u,\quad\text{ on }\mathbb{S}^2,
\end{equation*}
\begin{equation}\label{e:Yamabe}
-\Delta_{g_0}u + \frac{n - 2}{4(n-1)}R_{g_0}u = K u^{\frac{n + 2}{n - 2}},\quad\text{ on }\mathbb{S}^n,\quad n\geq 3,
\end{equation}
where $R_g$ is the scalar curvature of $g$.

The linear operator on the left of (\ref{e:Yamabe}) is known as the conformal Laplacian associated to the metric $g_0$ and is denoted as $P_1^{g_0}$. Another conformally covariant operator is
\begin{equation*}\label{e:Paneitz}
P_2^g = (-\Delta_g)^2 - {\rm div}_g (a_n R_gg + b_nRic_g)d + \frac{n -4}{2}Q_n^g,
\end{equation*}
which was discovered by Paneitz, see \cite{Paneitz2008} and \cite{DjadliHebeyLedoux2000}. Here $Q_n^g$, $Ric_g$ are the standard $Q$-curvature and the Ricci curvature of $g$ respectively, $a_n$, $b_n$ are constants depending on $n$. $P_1$ and $P_2$ (with $g$ be omitted when there is no ambiguity) are the first two terms of a sequence of conformally covariant elliptic operators $\{P_k\}$, which exists for all $k\in \mathbb{N}$ when $n$ is odd, but only for $k\in\{1,\cdot\cdot\cdot, n/2\}$ when $n$ is even.
The first construction of these operators was by Graham, Jenne, Masion and Sparling in \cite{GrahamJenneSparling1992}. Thus a natural question is: are there any conformally covariant pseudodifferential operators of noninteger orders? In \cite{Peterson2000}, the author constructed an intrinsically defined, arbitrary real number order, conformally covariant pseudo-differential operator. In the work of Graham and Zworski \cite{GramZworski2003}, it was showed that $P_k$ can be realized as the residues at $\gamma = k$ of a meromorphic family of
scattering operators. Using this view point, a family of conformally covariant
pseudodifferential operators $P_{\gamma}^g$ for noninteger $\gamma$ was given.

In recent years, there are extensive works on the properties of the fractional Laplacian as non-local operators together with their applications to various problems, for example, \cite{Caffarelli&Silvestre07}, \cite{CaffarelliRoquejoffreSavin}, \cite{CaffarelliSalsa&Silvestre}, \cite{CaffarelliValdinoci2011}, \cite{CaffarelliVasseurDFDQGE2010} and so on. It is well known that $(-\Delta)^\gamma$ on $\mathbb{R}^{n}$ with $\gamma\in (0, 1)$ is a nonlocal operator. In the remarkable work of Caffarelli and Silvestre \cite{Caffarelli&Silvestre07}, the authors express this nonlocal operator as a generalized Dirichlet-Neumann map for an elliptic boundary value problem with local differential operators defined on $\mathbb{R}^{n+1}_{+}$. And in the work of Chang and Gonzalez \cite{ChangGonzalez2011}, the authors extended the work of \cite{Caffarelli&Silvestre07} and characterized $P_{\gamma}^g$ as such a Dirichlet-to-Neumann operator on a conformally compact Einstein manifold.

The operator $P_\gamma^g$ with $\gamma\in (0, \frac{n}{4})$ has the following conformally covariant property:
if $g = v^{\frac{4}{n - 2\gamma}}g_0$, then
\begin{equation}\label{e:conformalltransform}
P_{\gamma}^{g_0}(vf) = v^{\frac{n + 2\gamma}{n - 2\gamma}}P_{\gamma}^g(f)
\end{equation}
for any smooth function $f$, see \cite{ChangGonzalez2011}. Generalizing the formula for scalar curvature and the Paneitz Branson $Q$-curvature, the $Q$-curvature for $g$ of order $\gamma$, is defined as
\begin{equation*}
Q_{\gamma}^{g} = P_{\gamma}^{g}(1).
\end{equation*}

In this paper, we are interested in the fractional Nirenberg problem on the standard sphere $\mathbb{S}^n$.
That is to say, we want to find a representation $g$ of the conformal class $[g_0]$ such that $Q_{\gamma}^{g}$ equals to a prescribed function $K$.
This problem is equivalent to solving the following semi-linear equation,
\begin{equation}\label{e:main}
\left
\{
\begin{array}{lll}
P_\gamma u = K u^{\frac{n + 2\gamma}{n - 2\gamma}}\text{ on }\mathbb{S}^{n},\\
u > 0,
\end{array}
\right.
\end{equation}
where $P_\gamma$ is the $2\gamma$ order conformal Laplacian on $\mathbb{S}^{n}$. This is an intertwining operator and
$$
P_\gamma = \frac{\Gamma(B + \frac{1}{2} +\gamma)}{\Gamma(B + \frac{1}{2} -\gamma)},\quad B = \sqrt{-\Delta_{g_0} + \left(\frac{n - 1}{2}\right)^2}.
$$
On the standard $n$ dimensional Riemannian sphere, the prescribing fractional curvature problem was considered in \cite{AbdelhediChtiouiNPISRLJFA2013}, \cite{ChenGuoyuanZhengYouquanNonlinearAnalysis}, \cite{ChenZhengFQ2014}, \cite{JinLiyanyanXiongNirenbergproblemBlowupanalysis}, \cite{JinLiyanyanXiongNirenbergproblemBlowupanalysisII} and \cite{JinXiongfYfsapplications2013}. On general manifolds, we refer the interested readers to the works \cite{GonzalezJGA2012}, \cite{GonzalezAPDE2013} and \cite{QingRaske2006} and the references therein.

Let $S^2_\gamma(\mathbb{S}^{n})$ be the completion of $C^\infty(\mathbb{S}^{n})$ by means of the norm
$$\|u\|^2 = \int_{\mathbb{S}^{n}}uP_\gamma u dvol_{g_0},$$
$$\Sigma = \{u\in S^2_\gamma(\mathbb{S}^{n})| \|u\| = 1\}$$
and
$$\Sigma^+ = \{u\in \Sigma| u\geq 0\}.$$
For $u\in S^2_\gamma(\mathbb{S}^{n})$, we consider the following functional
\begin{equation*}
J(u) = \frac{\|u\|^2}{\left(\int_{\mathbb{S}^{n}} K u^{\frac{2n}{n - 2\gamma}}dvol_{g_0}\right)^{\frac{n -2\gamma}{n}}}.
\end{equation*}
It is easy to see that a critical point of $J$ in $\Sigma^+$ corresponding to a solution of (\ref{e:main}). The functional $J$ fails to satisfy the Palais-Smale condition on $\Sigma^+$, a description of the sequences which do not satisfy the Palais-Smale condition is given in Lemma \ref{l:lossofcompactness} of Section 2. Thinking of these sequences as critical points, a natural idea is to expand the functional $J$ near the sets of such critical points.

We assume that $K: \mathbb{S}^{2n + 1}\to \mathbb{R}$ is a $C^2$ positive function and satisfies the following condition:
\begin{center}
({\bf nd}) {each critical point of $K$, denoted by $\xi$, is non degenerate, i.e., $\Delta K(\xi) \neq 0$}.
\end{center}
Denote
$$
I^+:=\{\xi_i\in \mathbb{S}^{2n + 1}| \nabla K(\xi_i) = 0\text{ and }-\Delta K(\xi_i) > 0\}
$$
and by $\sharp I^+$ the cardinality of $I^+$. Let $F^+$ be the set
$$
F^+ = \{(y_{i_1},\cdot\cdot\cdot, y_{i_p})\in (I^+)^p| y_{i_j}\neq y_{i_k}\text{ if }j\neq k, 1\leq p\leq \sharp I^+\}.
$$
In Section 5, we will prove that (Lemma \ref{l:criticalpointatinfinity} and Lemma \ref{l:morseindex})
\begin{lemma}\label{l:lemma1.1}
The critical points at infinity of $J$ (see Section 2 for its definition) in $\Sigma^+$ corresponding to
$$
\sum_{j = 1}^p\frac{1}{K(y_{i_j})^{\frac{n - 2\gamma}{4\gamma}}}\delta_{y_{i_j}, +\infty}: = (y_{i_1},\cdot\cdot\cdot, y_{i_p})_\infty,
$$
with $(y_{i_1},\cdot\cdot\cdot, y_{i_p})\in F^+$. The Morse index of such a critical point at infinity is
$$
ind((y_{i_1},\cdot\cdot\cdot, y_{i_p})_\infty) = p - 1 + \sum_{j = 1}^p(n - ind(K, y_{i_j})),
$$
where $ind(K, y_{i_j})$ is the Morse index of $K$ at $y_{i_j}$.
\end{lemma}

Let $F^+_\infty$ be the set of critical points at infinity of $J$. Then by Lemma \ref{l:lemma1.1} we have $$F^+_\infty = \{\phi_\infty^p := (y_{i_1},\cdot\cdot\cdot, y_{i_p})_\infty| (y_{i_1},\cdot\cdot\cdot, y_{i_p})\in F^+\}.$$ If $\phi_\infty^p\in F^+_\infty$, let $W_u(\phi_\infty^p)$ denote its unstable manifold and $W_s(\phi_\infty^p)$ its stable manifold, with respect to the $C^1$ vector field $-\partial J$. Then we have
$$
dim W_u(\phi_\infty^p) = codim W_s(\phi_\infty^p) = ind(\phi_\infty^p).
$$
For any $k\in \mathbb{N}$ and any subset $X_k$ of $\{\phi_\infty^p\in F^+_\infty: ind(\phi^p_\infty)\leq k\}$, we consider the following set
$$
X_k^\infty = \cup_{\phi_\infty^p\in X_k}W_u(\phi_\infty^p),
$$
which is a stratified set of dimension at most $k$. Since $\Sigma^+$ is a contractible set, $X_k^\infty$ is contractible in $\Sigma^+$ and let $\psi(X_k^\infty)$ be a contraction of $X_k^\infty$ in $\Sigma^+$.

Then we have
\begin{theorem}\label{t:main}
Assume that $n \geq 4$, $\gamma\in (0, 1)$ and $K$ satisfies condition ({\bf nd}). If there exists $k_0\in \mathbb{N}$ and $X_{k_0}\subset \{\phi_\infty^p\in F^+_\infty: ind(\phi_\infty^p)\leq k_0\}$ such that
\begin{enumerate}
\item[(H1)]
$$
\sum_{\phi_\infty^p\in X_{k_0}}(-1)^{ind(\phi_\infty^p)}\neq 1,
$$
\item[(H2)]
$$
\psi(X_{k_0}^\infty)\cap W_s(\phi_\infty^p) = \emptyset, \text{ for all }\phi_{\infty}^p\in F^+_\infty\setminus X_{k_0} \text{ with }ind(\phi_\infty^p)\leq k_0 + 1,
$$
\end{enumerate}
then there exits a solution $w$ of (\ref{e:main}) satisfying $ind(w)\leq k_0 + 1$.
\end{theorem}

As an application of Theorem \ref{t:main}, let $K$ assumes the following form
\begin{center}
{(\bf P)} $K(x) = 1 + \varepsilon K_0(x)$, $\forall x\in \mathbb{S}^n$, where $K_0\in C^2(\mathbb{S}^n)$ and $|\varepsilon|$ is small.
\end{center}
Set
\begin{equation}
\mathbb{T} = \{k\in \mathbb{N}|\forall y\in I^+, n - ind(K, y)\neq k + 1\}.
\end{equation}
Then we have
\begin{theorem}\label{t:main1}
Assume that $n \geq 4$, $\gamma\in (0, 1)$ and $K$ satisfies condition ({\bf nd}) and ({\bf P}). If
\begin{equation}\label{e:equation3}
\max_{k\in \mathbb{T}}\left|1 - \sum_{y\in I^+, n - ind(K, y)\leq k}(-1)^{n - ind(K, y)}\right|\neq 0,
\end{equation}
then when $|\varepsilon|$ is small, there exits a solution $w$ of (\ref{e:main}) satisfying $ind(w)\leq k_0 + 1$. Here $k_0$ achieves the maximum of (\ref{e:equation3}).
\end{theorem}
\begin{remark}
Theorem \ref{t:main} and \ref{t:main1} also holds when $n = 3$ and $\gamma \in (0, \frac{1}{2})$. This completes the study of Wael Abdelhedi and Hichem Chtioui \cite{AbdelhediChtiouiNPISRLJFA2013} in the  sense that, in this paper, the cases $n = 2$, $\gamma\in (0, 1)$ and $n = 3$, $\gamma\in (\frac{1}{2}, 1)$ were considered.
\end{remark}

We shall prove Theorem \ref{t:main} and \ref{t:main1} by contradiction, therefore we assume that (\ref{e:main}) has no solution. Our argument is based on a technical Morse lemma at infinity which involves the construction of a suitable pseudogradient for the function $J$ as in \cite{BahriPRNIMS1989, BahriDukeMathJYambetypeflows, BenChenChtiouiHamamiDukeMath1996, GamaraPSCCRANS2002, YacoubANS2013}. The Palais-Smale condition is satisfied along the decreasing flow lines of this pseudogradient, as long as these flow lines do not enter the neighborhood of a finite number of critical points of $K$. Finally, we obtain a Euler-Hopf type formula, this achieves a contradiction.

This paper is organized as follows. In Section 2, we introduce the general variational framework. In Section 3, we will give the expansion of the functional and its gradient near the sets of its critical points at infinity. In Section 4, we establish the Morse lemma at infinity, which allows us to refine the expansion of the function. In Section 5, we give the proof of Theorem \ref{t:main}. In Section 6, we give the proof of Theorem \ref{t:main1}. In Appendices A-C, we will give the estimates used in the proof.

\section{Variational structure}
Following \cite{BahriPRNIMS1989}, \cite{BahriCoronSCPSTDS1991} and \cite{BahriDukeMathJYambetypeflows}, we will use the following variational structure.
Consider the functional
\begin{equation*}
J(u) = \frac{\|u\|^2}{\left(\int_{\mathbb{S}^{n}} K u^{\frac{2n}{n - 2\gamma}}dvol_{g_0}\right)^{\frac{n -2\gamma}{n}}}
\end{equation*}
defined on $\Sigma$ which is the unit sphere of $S^{2}_{\gamma}(\mathbb{S}^{n})$. Let $\Sigma^+ = \{u\in \Sigma| u\geq 0\}$, problem (\ref{e:main}) will be reduced to finding critical points of $J$ subjected to the constraint $u\in\Sigma^+$. The exponent $\frac{2n}{n - 2\gamma}$ is critical for the Sobolev embedding $S^{2}_{\gamma}(\mathbb{S}^{n})\hookrightarrow L^q(\mathbb{S}^n)$. This embedding is continuous but not compact, so the functional $J$ does not satisfy the Palais Smale condition. This means that there exists a sequence along which $J$ is bounded, its gradient goes to zero but it does not converge. The characterization of sequences failing the Palais Smale condition can be analyzed along the ideas introduced in \cite{BahriPRNIMS1989}, \cite{BahriCoronSCPSTDS1991} and \cite{BahriDukeMathJYambetypeflows}. In order to describe such a characterization in our case, we need to introduce some notations.

For $a\in \mathbb{S}^n$ and $\lambda > 0$, let
$$
\delta_{a, \lambda}(x) = c_n\left(\frac{\lambda}{1 + \frac{\lambda^2 - 1}{2}(1 - cos(d(x, a)))}\right)^{\frac{n - 2\gamma}{2}},
$$
where $d(\cdot, \cdot)$ is the distance induced by the standard metric $g_0$, $c_n$ is chosen so that $\delta_{a, \lambda}$ is the family of the solutions for
\begin{equation}\label{e:yamabe2}
P_\gamma u = u^{\frac{n + 2\gamma}{n - 2\gamma}},\quad u > 0\quad\text{on}\quad\mathbb{S}^{n}.
\end{equation}
By the stereographic projection, (\ref{e:yamabe2}) can be transformed into the following equation
\begin{equation}\label{e:yamabe3}
(-\Delta)^\gamma u = u^{\frac{n + 2\gamma}{n - 2\gamma}},\quad u > 0\quad\text{on}\quad\mathbb{R}^{n}.
\end{equation}
And all positive regular solutions of (\ref{e:yamabe3}) are of form
$$
w_{g, \lambda}(x) = c_n\left(\frac{\lambda}{1 + \lambda^2|x - g|^2}\right)^{\frac{n - 2\gamma}{2}},
$$
see \cite{ChenWenxiongLicongmingOubiaoCPAM2006}, \cite{LiYanyan2004} and \cite{Lieb1983}.

For $p\geq 1$, we set
$$
\mathcal{E}^p = (0, +\infty)^p\times (\mathbb{S}^{n})^p\times (0, +\infty)^p,
$$
which is the space of the variables $(\alpha, a, \lambda) = (\alpha_1,\cdot\cdot\cdot, \alpha_p, a_1,\cdot\cdot\cdot, a_p, \lambda_1,\cdot\cdot\cdot, \lambda_p)$. For any small $\varepsilon > 0$ and $p\in \mathbb{N}^+$,  we will use the following subset of $\mathcal{E}^p$,
$$
B_\varepsilon = \left\{(\alpha, a, \lambda)\in \mathcal{E}^p|\varepsilon_{ij}\leq \varepsilon, \lambda_i > \frac{1}{\varepsilon}\right\},
$$
where
$$
\varepsilon_{ij} = \left(\frac{\lambda_i}{\lambda_j} + \frac{\lambda_j}{\lambda_i} + \lambda_i\lambda_jd(a_i, a_j)^2\right)^{-\frac{n - 2\gamma}{2}}.
$$
Now, we define the set $V(p, \varepsilon)$ of potential critical points at infinity to be
\begin{equation*}
V(p, \varepsilon) =\left\{u\in \Sigma\left|\right.
\begin{array}{lll}
\exists (\alpha, a, \lambda)\in B_\varepsilon, \text{such that } \|u - \sum_{i = 1}^p\alpha_i\delta_{a_i,\lambda_i}\| < \varepsilon\\
\text{ and } \left|J(u)^{\frac{n}{n - 2\gamma}}\alpha_j^{\frac{4\gamma}{n-2\gamma}}K(a_j) - 1\right| < \varepsilon
\end{array}
\right\}.
\end{equation*}

If $u$ is a function in $V(p, \varepsilon)$, one can find an optimal representation, following the ideas introduced in \cite{BahriPRNIMS1989}. Namely, we have
\begin{lemma}\label{l:decomposition}
For any $p\in\mathbb{N}^*$, there exists $\varepsilon_p > 0$ such that if $\varepsilon\leq \varepsilon_p$ and $u\in V(p, \varepsilon)$, then the following minimizing problem
\begin{equation}\label{e:minimization}
\min\left\{\|u - \sum_{i = 1}^p\alpha_i\delta_{a_i, \lambda_i}\|, \alpha_i > 0, a_i\in \mathbb{S}^n, \lambda_i > 0\right\}
\end{equation}
has a unique solution $(\bar{\alpha}, \bar{a}, \bar{\lambda})$. Thus, we can write $u$ as follows,
$$
u = \sum_{i = 1}^p\alpha_i\delta_{a_i, \lambda_i} + v
$$
where $v$ belongs to $S^{2}_{\gamma}(\mathbb{S}^{n})$ and satisfies the following condition
\begin{equation}\label{e:orthonormal}
\quad\quad\langle v, \varphi_i\rangle = 0\text{ for } i = 1,\cdot\cdot\cdot, p, \text{ and } \varphi_i  = \delta_{a_i, \lambda_i}, \frac{\partial \delta_{a_i, \lambda_i}}{\partial \lambda_i}, \frac{\partial \delta_{a_i, \lambda_i}}{\partial a_i}.
\end{equation}
Here $\langle\cdot, \cdot\rangle$ denotes the inner product in $S^{2}_{\gamma}(\mathbb{S}^{n})$ defined by
$$
\langle u, v\rangle = \int_{\mathbb{S}^{n}}vP_\gamma u dvol_{g_0}.
$$
\end{lemma}

Based on the uniqueness result of the corresponding problem at infinity (see \cite{ChenWenxiongLicongmingOubiaoCPAM2006} and \cite{LiYanyanJEMS2004}), the failure of the Palais Smale condition can be characterized following the ideas of \cite{BrezisCoronARMA1985}, \cite{GamaraPSCCRANS2002} and \cite{StruweMathZ1984}.
\begin{lemma}\label{l:lossofcompactness}
Assume that (\ref{e:main}) has no solution and let $\{u_k\}\subseteq \Sigma^+ $ be a sequence satisfying $J(u_k)\to c$, $J'(u_k)\to 0$. Then, there exist an integer $p \geq 1$, a positive sequence $\{\varepsilon_k\}$$(\varepsilon_k\to 0)$ and an extracted subsequence of $\{u_k\}$, still denoted by $\{u_k\}$, such that $u_k\in V(p, \varepsilon_k)$.
\end{lemma}

Following Bahri and Coron in \cite{BahriPRNIMS1989}, \cite{BahriCoronSCPSTDS1991} and \cite{BahriDukeMathJYambetypeflows}, we will use the following definition and notations later.

{\bf\noindent Definition} A critical point at infinity of $J$ on $\Sigma^+$ is a limit of a flow line $u(s)$ of the equation
\begin{equation*}
\left
\{
\begin{array}{lll}
\frac{\partial u(s)}{\partial s} = - J(u(s)),\\
u(0) = u_0,
\end{array}
\right.
\end{equation*}
such that $u(s)$ remains in $V(p, \varepsilon(s))$ for $s\geq s_0$, where $\varepsilon(s)\to 0$ as $s\to +\infty$ and $u_0$ is an initial condition.

Using Lemma \ref{l:decomposition}, $u(s)$ can be written as
$$
u(s) = \sum_{i = 1}^p\alpha_i(s)\delta_{a_i(s), \lambda_i(s)} + v(s).
$$
Let $a_i: = \lim_{s\to\infty}a_i(s)$ and $\alpha_i = \lim_{s\to\infty}\alpha_i(s)$, then such a critical point at infinity is denoted by
$$
(a_1,\cdot\cdot\cdot, a_p)_\infty\quad\text{ or }\quad \sum_{i = 1}^p\alpha_i\delta_{a_i, \infty}.
$$

\section{Expansion of the functional and its gradient}
In this section, we give the expansion of the functional and its gradient near the potential critical points at infinity.
\subsection{Expansion of the functional.}
\begin{lemma}\label{l:expansion}
For $\varepsilon > 0$ small enough and $u = \sum_{i = 1}^p\alpha_i\delta_{a_i, \lambda_i} + v\in V(p, \varepsilon)$ with $v$ satisfies condition (\ref{e:orthonormal}), we have the following equation,
\begin{eqnarray*}
J(u) &=& \frac{\sum_{i = 1}^p \alpha_i^2 S}{\left(\sum_{i = 1}^p \alpha_i^{\frac{2n}{n - 2\gamma}}K(a_i) S\right)^{\frac{n - 2\gamma}{n}}}\left[1 - \frac{n - 2\gamma}{n}\frac{c_2}{\Gamma_1}\sum_{i= 1}^p\alpha_i^{\frac{2n}{n - 2\gamma}}\frac{\Delta K(a_i)}{\lambda_i^2}\right]\\
&&+ \frac{\sum_{i = 1}^p \alpha_i^2 S}{\left(\sum_{i = 1}^p \alpha_i^{\frac{2n}{n - 2\gamma}}K(a_i) S\right)^{\frac{n - 2\gamma}{n}}}\left[\sum_{i\neq j}c_0^{\frac{2n}{n - 2\gamma}}c_1\omega_n\varepsilon_{ij}\left(\frac{\alpha_i\alpha_j}{\Gamma_2} -\frac{2\alpha_i^{\frac{n + 2\gamma}{n - 2\gamma}}\alpha_j K(a_i)}{\Gamma_1} \right)\right]\\
&&+ \frac{\sum_{i = 1}^p \alpha_i^2 S}{\left(\sum_{i = 1}^p \alpha_i^{\frac{2n}{n - 2\gamma}}K(a_i) S\right)^{\frac{n - 2\gamma}{n}}}\left[f(v) + Q(v, v) + o\left(\sum_{i\neq j}\varepsilon_{ij}\right) + o\left(\|v\|^2\right)\right],
\end{eqnarray*}
with
$$
f(v) = - \frac{2}{\Gamma_1}\int_{\mathbb{S}^{n}}K \left(\sum_{i = 1}^p\alpha_i\delta_{a_i, \lambda_i}\right)^{\frac{n + 2\gamma}{n - 2\gamma}}v\,\,dvol_{g_0},
$$
$$
Q(v, v) = \frac{1}{\Gamma_2}\|v\|^2 - \frac{2n + 4\gamma}{\Gamma_1(n + 2\gamma)}\sum_{i = 1}^p\int_{\mathbb{S}^{n}} K \left(\alpha_i\delta_{a_i, \lambda_i}\right)^{\frac{4\gamma}{n - 2\gamma}}v^2\,\,dvol_{g_0},
$$
$$
\Gamma_1 = \sum_{i = 1}^p \alpha_i^{\frac{2n}{n - 2\gamma}}K(a_i) S,\quad \Gamma_2 = \sum_{i = 1}^p \alpha_i^2 S.
$$
Furthermore, the operator norm of $f$ satisfies
$$
\|f\| = O\left(\sum_{i\neq j} \left(\varepsilon_{ij}\right)^{\frac{n + 2\gamma}{2n}} \left(\log \varepsilon_{ij}^{-1}\right)^{\frac{n - 2\gamma}{n}}+ \sum_{i = 1}^p\left(\frac{|\nabla K(a_i)|}{\lambda_i}+\frac{1}{\lambda_i^2}\right)\right).
$$
\end{lemma}
We will give the proof of this lemma in Appendix A.
\subsection{Expansion of the gradient.} In this subsection, we will compute the gradient of $J$ at $\lambda_j\frac{\partial \delta_{a_j, \lambda_j}}{\partial \lambda_j}$ and $\frac{1}{\lambda_j}\frac{\partial \delta_{a_j, \lambda_j}}{\partial a_j}$, respectively. We have
\begin{lemma}\label{l:expansionofgradient1}
Let $K$ be a $C^2$ positive function satisfying condition ({\bf nd}). Then for each $u = \sum_{i = 1}^p \alpha_i \delta_{a_i, \lambda_i}\in V(p, \varepsilon)$, we have the following expansion,
\begin{eqnarray*}
- J'(u) \left(\lambda_j\frac{\partial \delta_{a_j, \lambda_j}}{\partial \lambda_j}\right) &=& 2\lambda(u)\left[- \alpha_j\frac{n - 2\gamma}{2n}c_2\frac{\Delta K(a_j)}{K(a_j)\lambda_j^2} (1 + o(1)) \right]\\
&&+ 2\lambda(u)\left[\sum_{i\neq j}-\alpha_i\frac{4\gamma}{n + 2\gamma}c_0^{\frac{2n}{n-2\gamma}}c_1\omega_n\lambda_j\frac{\partial\varepsilon_{ij}}{\partial \lambda_j}(1 + o(1)) + o\left(\sum_{i\neq j}\varepsilon_{ij})\right)\right].
\end{eqnarray*}
\end{lemma}
and
\begin{lemma}\label{l:expansionofgradient2}
Let $K$ be a $C^2$ positive function satisfying condition ({\bf nd}). Then for each $u = \sum_{i = 1}^p \alpha_i \delta_{a_i, \lambda_i}\in V(p, \varepsilon)$, we have the following expansion,
\begin{eqnarray*}
&&-J'(u) \left(\frac{1}{\lambda_j}\frac{\partial \delta_{a_j, \lambda_j}}{\partial a_j}\right) = 2\lambda(u)\left[\frac{n - 2\gamma}{2n}c_0^{\frac{2n}{n-2\gamma}}c_1\omega_n\frac{\alpha_j}{K(a_j)}\frac{\nabla K(a_j)}{\lambda_j}(1 + o(1))\right.\\ && \quad\quad\quad\quad\quad\quad\quad\quad\quad\quad\quad\quad\quad\quad\quad\quad\quad\quad\quad\quad\quad\quad\quad\quad\quad\quad\quad\quad + \left.O\left(\sum_{i\neq j}\varepsilon_{ij} + \frac{1}{\lambda_j^3}\right)\right].
\end{eqnarray*}
\end{lemma}
We will give the proof of Lemma \ref{l:expansionofgradient1} in Appendix B and Lemma \ref{l:expansionofgradient2} in Appendix C.
\subsection{On the $v$-part of $u$.}
Set
$$
H_\varepsilon(a, \lambda) = \{v\in S^2_\gamma(\mathbb{S}^{n})|\,\, v \text{ satisfies } (\ref{e:orthonormal})\text{ and } \|v\|\leq \varepsilon\}.
$$
Then we have
\begin{lemma}\label{l:positivedefinite}
The quadratic form $Q$ in Lemma \ref{l:expansion} is positive definite in $H_\varepsilon(a, \lambda)$.
\end{lemma}
The proof is the same as the one in \cite{AbdelhediChtiouiNonlinearAnalysis2007} and uses the non degenerate result of \cite{DavilaDelpinoSireYannickPAMS2013}, so we omit it.
\begin{lemma}\label{l:vpartofu}
For any $u = \sum_{i = 1}^p \alpha_i \delta_{a_i, \lambda_i}\in V(p, \varepsilon)$, there exists a unique $\bar{v} = \bar{v}(\alpha, a, \lambda)$ which minimizes $J(u + v)$ with respect to $v\in E_\varepsilon$. Moreover, we have the following estimate
$$
\|\bar{v}\|\leq c\left[\sum_{i\neq j} \left(\varepsilon_{ij}\right)^{\frac{n + 2\gamma}{2n}} \left(\log \varepsilon_{ij}^{-1}\right)^{\frac{n - 2\gamma}{n}}+ \sum_{i = 1}^p\left(\frac{|\nabla K(a_i)|}{\lambda_i}+\frac{1}{\lambda_i^2}\right)\right].
$$
\end{lemma}
\begin{proof}
The proof of this lemma is similar to \cite{GamaraPSCCRANS2002} and \cite{YacoubANS2013}, for reader's convenience, we give it here. From Section 2, we know that the parameterization of $V(p, \varepsilon)$ is given by the following map
\begin{eqnarray*}
&&B_\varepsilon\times H_\varepsilon(a, \lambda)\to V(p, \varepsilon)
\end{eqnarray*}
\begin{eqnarray*}
&&(\alpha, a, \lambda, v)\to u= \sum_{i = 1}^p \alpha_i\delta_{a_i, \lambda_i} + v,
\end{eqnarray*}
where $(\alpha, a, \lambda)$ is the solution in $B_\varepsilon$ of the minimizing problem (\ref{e:minimization}), $v = u - \sum_{i = 1}^p\alpha_i\delta_{a_i, \lambda_i}\in H_\varepsilon(a, \lambda)$. Since $(\alpha, a, \lambda)\in B_\varepsilon$, $\varepsilon_{ij}$'s are small enough, then by Lemma \ref{l:positivedefinite}, the quadratic form $Q$ is definite positive in $H_\varepsilon(a, \lambda)$. Thus there exists a continuous self adjoint, positive definite and invertible operator $A$, such that $Q(v) = \frac{1}{2}\langle Av, v\rangle$ on $H_\varepsilon(a, \lambda)$ and $\beta_0 Id\leq A \leq \beta_1 Id$, here $\beta_1 > \beta_0 $ are positive constants. Then from Lemma \ref{l:expansion}, it holds that
\begin{eqnarray*}
J(u) &=& \frac{\sum_{i = 1}^p \alpha_i^2 S}{\left(\sum_{i = 1}^p \alpha_i^{\frac{2n}{n - 2\gamma}}K(a_i) S\right)^{\frac{n - 2\gamma}{n}}}\left[1 - \frac{n - 2\gamma}{n}\frac{c_2}{\Gamma_1}\sum_{i= 1}^p\alpha_i^{\frac{2n}{n - 2\gamma}}\frac{\Delta K(a_i)}{\lambda_i^2}\right]
\end{eqnarray*}
\begin{eqnarray*}
\quad\quad\quad\quad\quad\quad + \frac{\sum_{i = 1}^p \alpha_i^2 S}{\left(\sum_{i = 1}^p \alpha_i^{\frac{2n}{n - 2\gamma}}K(a_i) S\right)^{\frac{n - 2\gamma}{n}}}\left[\sum_{i\neq j}c_0^{\frac{2n}{n - 2\gamma}}c_1\omega_n\varepsilon_{ij}\left(\frac{\alpha_i\alpha_j}{\Gamma_2} -\frac{2\alpha_i^{\frac{n + 2\gamma}{n - 2\gamma}}\alpha_j K(a_i)}{\Gamma_1} \right)\right]
\end{eqnarray*}
\begin{eqnarray*}
\quad\quad\quad\quad\quad\quad + \frac{\sum_{i = 1}^p \alpha_i^2 S}{\left(\sum_{i = 1}^p \alpha_i^{\frac{2n}{n - 2\gamma}}K(a_i) S\right)^{\frac{n - 2\gamma}{n}}}\left[f(v) + \frac{1}{2}\langle Av, v\rangle + o\left(\sum_{i\neq j}\varepsilon_{ij}\right) + o\left(\|v\|^2\right)\right].
\end{eqnarray*}
Observe that the term $o(\|v\|^2)$ is, twice differentiable in $v$, and it's differential at the origin is $o(\|v\|)$. So the expansion of $J'$ along an increment $h$ near the origin in $H_\varepsilon(a, \lambda)$ is
\begin{eqnarray*}
\langle J'(u), h\rangle = \frac{\sum_{i = 1}^p \alpha_i^2 S}{\left(\sum_{i = 1}^p \alpha_i^{\frac{2n}{n - 2\gamma}}K(a_i) S\right)^{\frac{n - 2\gamma}{n}}}\left[f(h) + \langle Av, h\rangle + \langle o\left(\|v\|\right), h\rangle\right].
\end{eqnarray*}
Since the second differential of $o(\|v\|^2)$ is $o(1)$, the functional $f(v) + \frac{1}{2}\langle Av, v\rangle + o\left(\|v\|^2\right)$ is coercive in a neighborhood of the origin. Therefore, this functional has a unique minimum $\bar{v}$ in a neighborhood of zero in $H_\varepsilon(a, \lambda)$ and $\bar{v}$ satisfies
$$
f + A\bar{v} +o\left(\|\bar{v}\|\right) = 0.
$$
Now, since the operator $A + o(1)$ is positive and invertible in a neighborhood of the origin, the inverse $A^{-1}$ satisfy $\frac{2}{\beta_0} Id \geq A^{-1}\geq \frac{1}{2\beta_1}Id$ and $\|\bar{v}\|\leq c'\|A^{-1}f\|\leq c\|f\|$ for some constants $c$, $c' > 0$. This completes the proof.
\end{proof}
\begin{lemma}\label{l:change}
There exists $\varepsilon_0 > 0$ such that, for any $u = \sum_{i = 1}^p \alpha_i \delta_{a_i, \lambda_i} + v$, $v\in H_\varepsilon(a, \lambda)$, the following estimate holds
\begin{eqnarray*}
J(u) &=& \frac{\sum_{i = 1}^p \alpha_i^2 S}{\left(\sum_{i = 1}^p \alpha_i^{\frac{2n}{n - 2\gamma}}K(a_i) S\right)^{\frac{n - 2\gamma}{n}}}\left[1 + \frac{n - 2\gamma}{n}\frac{c_2}{\Gamma_1}\sum_{i= 1}^p\alpha_i^{\frac{2n}{n - 2\gamma}}\frac{-\Delta K(a_i)}{\lambda_i^2}\right]
\end{eqnarray*}
\begin{eqnarray*}
\quad\quad\quad\quad
&&+ \frac{\sum_{i = 1}^p \alpha_i^2 S}{\left(\sum_{i = 1}^p \alpha_i^{\frac{2n}{n - 2\gamma}}K(a_i) S\right)^{\frac{n - 2\gamma}{n}}}\left[\sum_{i\neq j}c_0^{\frac{2n}{n - 2\gamma}}c_1\omega_n\varepsilon_{ij}\left(\frac{\alpha_i\alpha_j}{\Gamma_2} -\frac{2\alpha_i^{\frac{n + 2\gamma}{n - 2\gamma}}\alpha_j K(a_i)}{\Gamma_1} \right)\right]
\end{eqnarray*}
\begin{eqnarray*}
\quad\quad\quad\quad
&&+ \frac{\sum_{i = 1}^p \alpha_i^2 S}{\left(\sum_{i = 1}^p \alpha_i^{\frac{2n}{n - 2\gamma}}K(a_i) S\right)^{\frac{n - 2\gamma}{n}}}\left[Q(v - \bar{v}, v - \bar{v}) + o\left(\|\bar{v}\|^2 + o\left(\sum_{i\neq j}\varepsilon_{ij}\right)\right)\right].
\end{eqnarray*}
\end{lemma}
\begin{proof}
Since $\bar{v}$ is a minimizer, we have
$$
(f, \bar{v}) + Q(\bar{v},\bar{v}) + o(\|\bar{v}\|^2) = 0.
$$
This yields
$$
(f, v) + Q(v, v) + o(\|v\|^2) = Q(v - \bar{v}, v - \bar{v}) + o(\|\bar{v}\|^2).
$$
From this, we get the desired estimate.
\end{proof} 
\section{Morse lemma at infinity}
In this section, we prove the following Morse lemma, which completely get rid of the $v$-contributions and shows that the functional behaves, at infinity, as $J(\sum_{i= 1}^p\alpha_i\delta_{\tilde{a}_i, \tilde{\lambda}_i}) + |V|^2$, where $V$ is a variable completely independent of $\tilde{a_i}$, $\tilde{\lambda}_i$.
\begin{lemma}\label{l:morselemma}
There is a covering $\{O_l\}$, a subset $\{(\alpha_l,, a_l, \lambda_l)\}$ of the base space for the bundle $V(p, \varepsilon)$ and a diffeomorphism $\xi_l: V(p, \varepsilon)\to V(p, \varepsilon')$ for some $\varepsilon' > 0$ with
$$
\xi_l(\sum_{i= 1}^p\alpha_i\delta_{a_i, \lambda_i} + \bar{v}) = \sum_{i= 1}^p\alpha_i\delta_{\tilde{a}_i, \tilde{\lambda}_i}
$$
such that
$$
J(\sum_{i= 1}^p\alpha_i\delta_{a_i, \lambda_i} + v) = J(\sum_{i= 1}^p\alpha_i\delta_{\tilde{a}_i, \tilde{\lambda}_i}) + \frac{1}{2}J''(\sum_{i= 1}^p\alpha_i\delta_{a_i, \lambda_i})V_l\cdot V_l,
$$
where $(\alpha, a, \lambda)\in O_l$, $(\alpha, \tilde{a}, \tilde{\lambda})$ is independent of $O_l$ and $V_l$ is orthogonal to $\delta_{\tilde{a}_i, \tilde{\lambda}_i}$, $\frac{\partial \delta_{\tilde{a}_i, \tilde{\lambda}_i}}{\partial \tilde{\lambda}_i}$, $\frac{\partial \delta_{\tilde{a}_i, \tilde{\lambda}_i}}{\partial \tilde{a}_i}$.
\end{lemma}

We will prove this lemma at the end of this section. We now need a few technical results. We start with the Morse lemma at infinity by isolating the contribution of $v - \bar{v}$.
\begin{lemma}\label{l:taylorformular}
For any $\sum_{i= 1}^p\bar{\alpha}_i\delta_{\bar{a}_i, \bar{\lambda}_i}\in V(p, \varepsilon)$, let $(\bar{\alpha}, \bar{a}, \bar{\lambda}) = ((\bar{\alpha}_1, \cdot\cdot\cdot, \bar{\alpha}_p), (\bar{a}_1, \cdot\cdot\cdot, \bar{a}_p), (\bar{\lambda}_1, \cdot\cdot\cdot, \bar{\lambda}_p))$. Then there is a neighborhood $U$ of $(\bar{\alpha}, \bar{a}, \bar{\lambda})$ such that
$$
J(\sum_{i= 1}^p\alpha_i\delta_{a_i, \lambda_i} + v) = J(\sum_{i= 1}^p\alpha_i\delta_{a_i, \lambda_i} + \bar{v}(\alpha, a, \lambda)) + \frac{1}{2}J''(\sum_{i= 1}^p\bar{\alpha}_i\delta_{\bar{a}_i, \bar{\lambda}_i} + \bar{v}(\bar{\alpha}, \bar{a}, \bar{\lambda}))V\cdot V
$$
for any $\sum_{i= 1}^p\alpha_i\delta_{a_i, \lambda_i} + v\in V(p, \varepsilon)$ with $(\alpha, a, \lambda)\in U$, where $V = V(\alpha, a, \lambda, v)$ is a $C^1$-diffeomorphism with range orthogonal to
$$
\displaystyle\cup_{i = 1}^p\left\{\delta_{a'_i, \lambda'_i}, \frac{\partial \delta_{a'_i, \lambda'_i}}{\partial \lambda'_i}, \frac{\partial \delta_{a'_i, \lambda'_i}}{\partial a'_i}\right\}
$$
for any $(\alpha', a', \lambda')\in U$ and $\|V\| = O(\|v\|)$.
\end{lemma}
The proof is similar to the one given for the Riemannian case, we refer the readers to \cite{BenChenChtiouiHamamiDukeMath1996} for the sake of completeness.

\begin{lemma}\label{l:characterization}
For any $u = \sum_{i = 1}^p\alpha_i \delta_{a_i, \lambda_i}\in V(p, \varepsilon')$, $\varepsilon'$ small enough, let $\bar{u} = u + \bar{v}(\alpha, a, \lambda)$. There is a vector field $W_0$ such that for some constants $C_1 > 0$, $C_2 > 0$, $C_3 > 0$ independent of $u = \sum_{i = 1}^p\alpha_i \delta_{a_i, \lambda_i}\in V(p, \varepsilon')$, it holds that
\begin{enumerate}
\item[(1)]
$$
\langle - \nabla J(\bar{u}),\,\, W_0 + \frac{\partial \bar{v}}{\partial (\alpha_i, a_i, \lambda_i)}(W_0)\rangle\geq C_1\left(\sum_{i = 1}^p\left(\frac{|\nabla K(a_i)|}{\lambda_i} + \frac{1}{\lambda_i^2}\right) + \sum_{i\neq j} \varepsilon_{ij}\right),
$$
\item[(2)]
$$
\langle - \nabla J(u), W_0\rangle\geq C_2\left(\sum_{i = 1}^p\left(\frac{|\nabla K(a_i)|}{\lambda_i} + \frac{1}{\lambda_i^2}\right) + \sum_{i\neq j} \varepsilon_{ij}\right),
$$
\item[(3)]
$|W_0|$ is bounded,
\item[(4)]
$d\lambda_{i}(W_0)\leq C_3 \lambda_{i}$,\quad $\forall i\in\{1,\cdot\cdot\cdot, p\}$,
\item[(5)]
The only region where the $\lambda_i$'s are not bounded along the decreasing flow lines of $W_0$ is where $(a_1,\cdot\cdot\cdot, a_p)$ is close to some $(y_{i_1}, \cdot\cdot\cdot, y_{i_p})\in F^+$, and the $\lambda_i$'s are comparable.
\end{enumerate}
\end{lemma}
\begin{proof} We follow the proof of \cite{BahriDukeMathJYambetypeflows} and \cite{YacoubANS2013}. We need to define $W_0$ so that the Palais Smale condition is satisfied on its decreasing flow lines and $W_0$ has no action on the $\alpha_i$'s variables. Moving the $a_i$'s contains no risk for the Palais Smale condition, since they lie in a compact set, so we only need to prove
$$
\forall s\geq 0,\quad \frac{\partial }{\partial s}(\sup_{1\leq i\leq p}\lambda_i)\leq 0,
$$
where $s$ is the time along a flow line of $W_0$.

Since $|\lambda_i\frac{\partial \varepsilon_{ij}}{\partial a_i}|\leq c\varepsilon_{ij}$, we derive from Lemma \ref{l:expansionofgradient2} that
\begin{equation*}
|J'(\sum_j \alpha_j\delta_{a_j, \lambda_j})\frac{1}{\lambda_i}\frac{\partial \delta_{a_i,\lambda_i}}{\partial a_i}|\geq c\frac{|\nabla K(a_i)|}{\lambda_i} - \frac{1}{c}(\sum_{j\neq i}\varepsilon_{ij} + \frac{1}{\lambda_i^2}),
\end{equation*}
where $c$ is a positive constant.

If for all $i = 1,\cdot\cdot\cdot, p$, it holds that
\begin{equation}\label{e:assumption1}
\sum_{j\neq i}\varepsilon_{ij}\leq \frac{C}{\lambda_i^2},
\end{equation}
where $C$ is a suitable constant, then we have
\begin{equation*}
|J'(\sum_j \alpha_j\delta_{a_j, \lambda_j})\frac{1}{\lambda_i}\frac{\partial \delta_{a_i,\lambda_i}}{\partial a_i}|\geq c\frac{|\nabla K(a_i)|}{\lambda_i} - \frac{1}{c'}\frac{1}{\lambda_i^2}
\end{equation*}
for a suitable constant $c'$.

If (\ref{e:assumption1}) does not hold for some index, we choose the index $i$ so that $\lambda_i$ is the largest concentration with
\begin{equation}\label{e:assumption2}
\sum_{j\neq i}\varepsilon_{ij} > \frac{C}{\lambda_i^2}.
\end{equation}
Then for $\lambda_j\geq \lambda_i$, we have
$$
\sum_{k\neq j}\varepsilon_{kj} \leq \frac{C}{\lambda_j^2}.
$$
Observe that, if $\lambda_j$ and $\lambda_i$ are comparable, or if $\lambda_i\geq \lambda_j$, then
$$
\lambda_i\frac{\partial \varepsilon_{ij}}{\partial \lambda_i} = -n\varepsilon_{ij}(1 + o(1)).
$$
If they are not and $\lambda_i = o(\lambda_j)$, then
$$
\lambda_i\frac{\partial \varepsilon_{ij}}{\partial \lambda_i} = O(\varepsilon_{ij})\leq \frac{C}{\lambda_j^2} = o(\frac{1}{\lambda_i^2}).
$$
Thus we have
$$
-\sum_{j\neq i}\lambda_i\frac{\partial \varepsilon_{ij}}{\partial \lambda_i} \geq \frac{n}{2}\sum_{j\neq i}\varepsilon_{ij} \geq \frac{C}{2\lambda_i^2}.
$$
Hence, choosing $C$ large enough, it holds that
\begin{equation*}
J'(\sum_j \alpha_j\delta_{a_j, \lambda_j})\lambda_i\frac{\partial \delta_{a_i, \lambda_i}}{\partial \lambda_i}\geq \frac{C}{4\lambda_i^2}.
\end{equation*}
Combine the estimates above, we have
\begin{equation*}\label{e:gradient1}
|J'(\sum_j \alpha_j\delta_{a_j, \lambda_j})\frac{1}{\lambda_i}\frac{\partial \delta_{a_i,\lambda_i}}{\partial a_i}| + \frac{\tilde{C}}{C}J'(\sum_j \alpha_j\delta_{a_j, \lambda_j})\lambda_i\frac{\partial \delta_{a_i, \lambda_i}}{\partial \lambda_i}\geq c\frac{|\nabla K(a_i)|}{\lambda_i} - \frac{1}{c}\frac{1}{\lambda_i^2}
\end{equation*}
for suitable positive constants $C$, $\tilde{C}$.

Assume now, that we have another index $i$ such that (\ref{e:assumption2}) holds, but $\lambda_i$ is not the largest. We introduce the set
$$
I_i = \{k| \lambda_k\geq \lambda_i, \sum_{k\neq j}\varepsilon_{kj} > \frac{C}{\lambda_k^2}\}.
$$
Observe that for $\lambda_k\geq \lambda_j$, we have
$$
2\lambda_k\frac{\varepsilon_{kj}}{\lambda_k} + \lambda_j\frac{\partial \varepsilon_{kj}}{\partial \lambda_j}\leq -\frac{n}{2}\varepsilon_{ij}(1 + o(1)).
$$
By similar arguments as above , we derive the existence of suitable bounded constants $c_k$ such that
\begin{equation*}
|J'(\sum_j \alpha_j\delta_{a_j, \lambda_j})\frac{1}{\lambda_i}\frac{\partial \delta_{a_i,\lambda_i}}{\partial a_i}| + \sum_{k\in I_i}c_k\frac{\tilde{C}}{C}J'(\sum_j \alpha_j\delta_{a_j, \lambda_j})\lambda_k\frac{\partial \delta_{a_k, \lambda_k}}{\partial \lambda_k}\geq c\frac{|\nabla K(a_i)|}{\lambda_i} - \frac{1}{c}\frac{1}{\lambda_i^2}.
\end{equation*}

We order the concentrations as follows,
$$
\lambda_1\leq \cdot\cdot\cdot\leq \lambda_p.
$$
If
$$
\frac{|\nabla K(a_1)|}{\lambda_1} \geq \frac{1}{c^2}\frac{1}{\lambda_1^2},
$$
then
\begin{equation*}
|J'(\sum_j \alpha_j\delta_{a_j, \lambda_j})\frac{1}{\lambda_1}\frac{\partial \delta_{a_1, \lambda_1}}{\partial a_1}| + \sum_{k\in I_1}c_k\frac{\tilde{C}}{C}J'(\sum_j \alpha_j\delta_{a_j, \lambda_j})\lambda_k\frac{\partial \delta_{a_k, \lambda_k}}{\partial \lambda_k}\geq c\frac{|\nabla K(a_i)|}{4\lambda_i} + \frac{1}{c}\frac{1}{4\lambda_i^2}.
\end{equation*}
This yields that
\begin{eqnarray*}
&&|J'(\sum_j \alpha_j\delta_{a_j, \lambda_j})\frac{1}{\lambda_1}\frac{\partial \delta_{a_1, \lambda_1}}{\partial a_1}| + |J'(\sum_j \alpha_j\delta_{a_j, \lambda_j})\frac{1}{\lambda_i}\frac{\partial \delta_{a_i,\lambda_i}}{\partial a_i}| + J'(\sum_j \alpha_j\delta_{a_j, \lambda_j})\lambda_i\frac{\partial \delta_{a_i, \lambda_i}}{\partial \lambda_i}\\ &&\quad+ \sum_{k\in I_1}c_k\frac{\tilde{C}}{C}J'(\sum_j \alpha_j\delta_{a_j, \lambda_j})\lambda_k\frac{\partial \delta_{a_k, \lambda_k}}{\partial \lambda_k} + \sum_{k\in I_i}c_k\frac{\tilde{C}}{C}J'(\sum_j \alpha_j\delta_{a_j, \lambda_j})\lambda_k\frac{\partial \delta_{a_k, \lambda_k}}{\partial \lambda_k}\\
&&\quad\geq -\sum_{j\neq i}c_{ij}\lambda_i\frac{\partial \varepsilon_{ij}}{\partial \lambda_i}  + \frac{|\nabla K(a_i)|}{\lambda_i} + \frac{\bar{c}}{\lambda_i^2}.
\end{eqnarray*}
So there exist nonnegative constants $\beta_i$, $\gamma_i$ such that
\begin{eqnarray*}
&&\sum_i \beta_i|J'(\sum_j \alpha_j\delta_{a_j, \lambda_j})\frac{1}{\lambda_i}\frac{\partial \delta_{a_i,\lambda_i}}{\partial a_i}| + J'(\sum_j \alpha_j\delta_{a_j, \lambda_j})\sum_i\gamma_i\lambda_i\frac{\partial \delta_{a_i, \lambda_i}}{\partial \lambda_i}\\ &&\quad\geq \bar{c}(\sum_i\frac{|\nabla K(a_i)|}{\lambda_i} + \frac{1}{\lambda_i^2} + \sum_{j\neq i}\varepsilon_{ij}).
\end{eqnarray*}
Moreover, $\beta_i$ can be chosen such that
$$
\beta_i = 0\quad \text{if}\quad |J'(\sum_j \alpha_j\delta_{a_j, \lambda_j})\frac{1}{\lambda_i}\frac{\partial \delta_{a_i,\lambda_i}}{\partial a_i}| < \frac{\bar{c}}{10p}\sum_i\frac{1}{\lambda_i^2}.
$$
Define
$$
W_0 = -\sum_i \beta_i \text{sign} (J'(\sum_j \alpha_j\delta_{a_j, \lambda_j})\frac{1}{\lambda_i}\frac{\partial \delta_{a_i,\lambda_i}}{\partial a_i})\frac{1}{\lambda_i}\frac{\partial \delta_{a_i,\lambda_i}}{\partial a_i} - \sum_i\gamma_i\lambda_i\frac{\partial \delta_{a_i, \lambda_i}}{\partial \lambda_i}.
$$
Then $W_0$ is a $C^1$ vector field and $\|W_0\|\leq C$. Since $d\lambda_i(W_0) = -\gamma_i\lambda_i$, thus $|d\lambda_i(W_0)|\leq C\lambda_i$, $\forall i = 1,\cdot\cdot\cdot, p$.

A similar proof can be repeated if we assume
$$
\sum_{j\neq i}\varepsilon_{ij} \geq \frac{C}{\lambda_1^2}.
$$

The above proof can be extended as follows. Assume that instead of $\lambda_1\leq \cdot\cdot\cdot\leq \lambda_p$, we single out a subsequence
$$
\lambda_{i_1}\leq \lambda_{i_1 + 1}\leq\cdot\cdot\cdot\leq \lambda_{i_1 + r}.
$$
We will construct a vector field $W_{(i_1, r)}$ in $span\{\frac{\partial \delta_{a_i, \lambda_i}}{\partial a_i}, \frac{\partial \delta_{a_i, \lambda_i}}{\partial \lambda_i}\}_{i_1\leq i\leq i_1 + r}$ such that
$$
\|W_{(i_1, r)}\|\leq C
$$
and
$$
0\leq -d(\sup_{i_1\leq i\leq i_1 + r}\lambda_i)(W_{(i_1, r)})\leq C \sup_{i_1\leq i\leq i_1 + r}\lambda_i.
$$

Under the assumption
\begin{equation}\label{e:assumption3}
\frac{|\nabla K(a_{i_1})|}{\lambda_{i_1}} \geq \frac{2}{c^2}\frac{1}{\lambda_{i_1}^2}\quad\text{or}\quad \sum_{s = 0}^r\sum_{j\leq i_1 + r, j\neq i_1 + s}\varepsilon_{i_1 + s, j}\geq \frac{1}{\lambda_{i_1}^2}
\end{equation}
we have a vector field $W_{(i_1, r)}$ such that
\begin{eqnarray*}
-J'(\sum_j\alpha_j\delta_{a_j, \lambda_j})(W_{(i_1, r)})&\geq & C(\sum_{s = 0}^r\sum_{j\leq i_1 + r, j\neq i_1 + s}\varepsilon_{i_1 + s, j} + \sum_{s = 0}^r \frac{1}{\lambda_{i_1 + s}^2}\\ &&+ \sum_{s = 0}^r \frac{|\nabla K(a_{i_1 + s})|}{\lambda_{i_1 + s}} - \frac{1}{\bar{c}}\sum_{j\geq i_1 + r + 1}\sum_{s = 0}^r\varepsilon_{i_1 + s, j}).
\end{eqnarray*}

We first assume that such indices $i_1$ satisfying (\ref{e:assumption3}) exist and we assume $i_1$ is the the smallest concentration satisfying (\ref{e:assumption3}). Since $\lambda_{i_1}\leq \cdot\cdot\cdot\leq \lambda_p$, we derive there exists a vector field $W_{(i_1, p - i_1)}$ such that
\begin{eqnarray}\label{e:assumption5}
-J'(\sum_j\alpha_j\delta_{a_j, \lambda_j})(W_{(i_1, p - i_1)})&\geq & C(\sum_{i\geq i_1} \frac{|\nabla K(a_{i})|}{\lambda_{i}} + \sum_{i\geq i_1} \frac{1}{\lambda_{i}^2} + \sum_{k = i_1}^p\varepsilon_{ik}).
\end{eqnarray}
If $i_1 = 1$, we have the result. Otherwise, for any $l < i_1$, it holds that
\begin{equation}\label{e:assumption4}
\frac{|\nabla K(a_{l})|}{\lambda_{l}} \leq \frac{2}{c^2}\frac{1}{\lambda_{l}^2}\quad\text{and}\quad \sum_{k = l}^p\sum_{i\neq k}\varepsilon_{ik}\leq \frac{c}{\lambda_{l}^2}.
\end{equation}
It is easy to see that the desired estimate follows from (\ref{e:assumption5}) and (\ref{e:assumption4}), unless:
\begin{equation}\label{e:assumption6}
\sum_{i\geq i_1} \frac{|\nabla K(a_{i})|}{\lambda_{i}} + \sum_{i\geq i_1} \frac{1}{\lambda_{i}^2} + \sum_{k = i_1}^p\varepsilon_{ik} = o(\frac{1}{\lambda_l^2}) = o(\frac{1}{\lambda_1^2})
\end{equation}
for some $l\leq i_1 - 1$. Assume that (\ref{e:assumption6}) holds, then for $i\leq i_1 - 1$, one has $\lambda_i|\nabla K(a_i)|\leq \frac{2}{c^2}$
and $|\nabla K(a_i)| = o(1)$. So, for $i\leq i_1 - 1$, $a_i$ is close to a critical point of $K$ which we denote by $y_i$. so $\lambda_i d(a_i, y_i)\leq C$ for $i \leq i_1 - 1$.  Consequently, if for $i, j\leq i_1 - 1$, $a_i$ and $a_j$ are close to the same critical point $y_i$, $\frac{\lambda_i}{\lambda_j}\to + \infty$ or $\frac{\lambda_j}{\lambda_i}\to +\infty$, and $\varepsilon_{ij} = o(\frac{1}{\inf(\lambda_i,\lambda_j)^2}) = o(\frac{1}{\lambda_1^2})$. Now, if $a_i$ and $a_j$ are close to distinct critical points $y_i$ and $y_j$, $\varepsilon_{ij} = o(\frac{1}{\lambda_1^2})$. Thus, for all $i, j\leq i_1 - 1$, the $\varepsilon_{ij}$'s are $o(\frac{1}{\lambda_1^2})$. From (\ref{e:assumption6}), this fact also holds for $i, j\geq i_1$. Thus we have
$$
\sum_{i\neq j}\varepsilon_{ij} = o(\frac{1}{\lambda_1^2}).
$$
This implies
$$
\sum_{i\geq i_1} \frac{|\nabla K(a_{i})|}{\lambda_{i}} = o(\frac{1}{\lambda_1^2}).
$$
But
$$
\sum_{i < i_1} \frac{|\nabla K(a_{i})|}{\lambda_{i}} \leq \sum_{i < i_1}\frac{c}{\lambda_i^2}\leq \frac{c}{\lambda_1^2}
$$
and so
$$
\sum_{i}^p \frac{|\nabla K(a_{i})|}{\lambda_{i}} \leq \frac{c}{\lambda_1^2}.
$$
Hence it holds that
\begin{eqnarray*}\label{e:assumption10}
-J'(\sum_j\alpha_j\delta_j)\lambda_1\frac{\partial \delta_{a_1, \lambda_1}}{\partial \lambda_1} &=& -c\frac{-\Delta K(y_1)}{\lambda_1^2} + o(\frac{1}{\lambda_1^2}) + \sum_{i\neq 1}c_{1,i}\lambda_1\frac{\partial \varepsilon_{1i}}{\partial \lambda_1}\\
&=& -c\frac{-\Delta K(y_1)}{\lambda_1^2} + o(\frac{1}{\lambda_1^2}).
\end{eqnarray*}
Furthermore, it holds that
$$
\frac{1}{\lambda_1^2}\geq c^2(\sum_{i}^p \frac{|\nabla K(a_{i})|}{\lambda_{i}} + \sum_{i = 1}^p \frac{1}{\lambda_i^2}+ \sum_{i\neq j}\varepsilon_{ij}),
$$
for some constant $c > 0$.

If
$$
-\Delta K(y_1) \geq c' > 0,
$$
then
\begin{equation*}\label{e:assumption7}
-J'(\sum_j\alpha_j\delta_j)\lambda_1\frac{\partial \delta_{a_1, \lambda_1}}{\partial \lambda_1}\geq c(\sum_{i}^p \frac{|\nabla K(a_{i})|}{\lambda_{i}} + \sum_{i = 1}^p \frac{1}{\lambda_i^2}+ \sum_{i\neq j}\varepsilon_{ij}).
\end{equation*}
Set
$$
W_0 = \lambda_1\frac{\partial \delta_{a_1, \lambda_1}}{\partial \lambda_1}.
$$
Then, $d(\sup_i \lambda_i)(W_0)\leq 0$.

If
$$
-\Delta K(y_1) < 0,
$$
set
$$
W_0 = -\lambda_1\frac{\partial \delta_{a_1, \lambda_1}}{\partial \lambda_1}
$$
Then it satisfies all the required properties. Lemma \ref{l:characterization} then follows as soon as $i_1$ exists.

Assume now that such $i_1$ satisfying (\ref{e:assumption3}) does not exist, that is to say, $\forall i\in \{1, \cdot\cdot\cdot, p\}$,
\begin{equation*}\label{e:assumption8}
\frac{|\nabla K(a_{i})|}{\lambda_{i}} \leq \frac{2}{c^2}\frac{1}{\lambda_{i}^2}\quad\text{and}\quad \sum_{k = l}^p\sum_{i\neq k}\varepsilon_{ik}\leq \frac{c}{\lambda_{i}^2}.
\end{equation*}
We assume that
$$
\inf d(a_i, a_j) < \frac{1}{2}\inf d(y_k, y_l).
$$
(otherwise, the proof is straightforward). Under this condition, $a_i$ and $a_j$ are close to some same critical point $y_i$, then $\inf(\lambda_i, \lambda_j) = o(\sup(\lambda_i, \lambda_j))$, so
\begin{equation*}\label{e:assumption9}
\frac{|\nabla K(a_{i})|}{\lambda_{i}} \leq \frac{2}{c^2}\frac{1}{\lambda_{1}^2}\quad\text{and}\quad \sum_{k = l}^p\sum_{i\neq k}\varepsilon_{ik}\leq o(\frac{1}{\lambda_{1}^2}).
\end{equation*}
Then the same argument used in the previous case can be repeated.

Since the same argument is valid when two concentrations are not comparable. So we will assume now that $\inf d(a_i, a_j)\geq d_0 > 0$, and all concentrations are comparable, that is, $\frac{1}{c}\leq \frac{\lambda_i}{\lambda_j}\leq c$. If some index $i$, $a_i$ is not close to some critical point $y_i$, then $\varepsilon_{ij} = o(\frac{1}{\lambda_i^2}) = o(\frac{1}{\lambda_j^2})$, and $\frac{|\nabla K(a_i)|}{\lambda_i}\geq \frac{1}{\lambda_i^2}$, so $|J'(\sum_j \alpha_j\delta_{a_j, \lambda_j})\frac{1}{\lambda_i}\frac{\partial \delta_{a_i,\lambda_i}}{\partial a_i}|\geq c\frac{|\nabla K(a_i)|}{\lambda_i} - \frac{1}{c}\frac{1}{\lambda_i^2}\geq \frac{c}{2\lambda_i^2}$, and
$$
|J'(\sum_j \alpha_j\delta_{a_j, \lambda_j})\frac{1}{\lambda_i}\frac{\partial \delta_{a_i,\lambda_i}}{\partial a_i}|\geq c(\sum_{k = 1}^p\frac{|\nabla K(a_k)|}{\lambda_k} + \sum_{k = 1}^p \frac{1}{\lambda_k^2} + \sum_{k\neq l}\varepsilon_{kl}).
$$
Hence $W_0 = \frac{1}{\lambda_i}\frac{\partial \delta_{a_i,\lambda_i}}{\partial a_i}$ is the desired vector field.
Now, we are left with the case where each point $a_i$ is close to a critical point $y_i$. It holds that
$$
-J'(\sum_j\alpha_j\delta_j)\lambda_i\frac{\partial \delta_{a_i, \lambda_i}}{\partial \lambda_i} = - c\frac{\Delta K(y_i)}{\lambda_i^2} + o(\frac{1}{\lambda_i^2}).
$$
Set $W_0 = \lambda_i\frac{\partial \delta_{a_i, \lambda_i}}{\partial \lambda_i}$. If $-\Delta K(y_i) < 0$, $d\lambda_i(W_0) = -\lambda_i$, so $d(\sup_i\lambda_i)(W_0)\leq 0$, and in this case $w_0$ is constructed.

The final pseudogradient vector field $W_0$ will be a convex combination of the vector field constructed in the above cases.

Finally, it remains the case where the points $a_i$ are close to distinct critical points $y_i$, having $-\Delta K(y_i) > 0$ and all the concentrations are comparable. In this region the decreasing flow lines of $W_0$ are attracted by the critical point at infinity $(y_{j_1},\cdot\cdot\cdot, y_{j_p})_\infty$. Thus condition (5) is satisfied and this completes the proof.
\end{proof}
\begin{lemma}\label{l:lemma4.4}
For any $u = \sum_{i = 1}^p\alpha_i \delta_{a_i, \lambda_i}\in V(p, \varepsilon_1)$ $(\varepsilon_1 < \varepsilon/2)$, we have
$$
J(\sum_{i= 1}^p\alpha_i\delta_{a_i, \lambda_i} + \bar{v}(\alpha, a, \lambda)) = J(\sum_{i= 1}^p\alpha_i\delta_{\tilde{a}_i, \tilde{\lambda}_i})
$$
with
$$
\sum_{i\neq j}\tilde{\varepsilon}_{ij} + \sum_{i}\frac{1}{\tilde{\lambda}_i^2} \to 0 \Leftrightarrow \sum_{i\neq j}\varepsilon_{ij} + \sum_{i}\frac{1}{\lambda_i^2} \to 0
$$
and
$$
|\tilde{a}_i - a_i|\to 0\text{ as } \sum_{i\neq j}\varepsilon_{ij} + \sum_{i}\frac{1}{\lambda_i^2} \to 0.
$$
\end{lemma}
\begin{proof}
We follow the proof of \cite{BenChenChtiouiHamamiDukeMath1996} and \cite{GamaraPSCCRANS2002}. By Lemma \ref{l:characterization}, the vector field $W_0$ is Lipschitz. Hence, there is a one parameter group $h_s$ generated by $W_0$ satisfying
\begin{equation*}
\left\{
             \begin{array}{lr}
             \frac{\partial }{\partial s}h_s(\sum_{i = 1}^p\alpha_i \delta_{a_i, \lambda_i}) = W_0(h_s(\sum_{i = 1}^p\alpha_i \delta_{a_i, \lambda_i})),&  \\
             h_0(\sum_{i = 1}^p\alpha_i \delta_{a_i, \lambda_i}) = \sum_{i = 1}^p\alpha_i \delta_{a_i, \lambda_i}.
             \end{array}
             \right.
\end{equation*}
Therefore $J(h_s(\sum_{i = 1}^p\alpha_i \delta_{a_i, \lambda_i}))$, $J(h_s(\sum_{i = 1}^p\alpha_i \delta_{a_i, \lambda_i}) + \bar{v}(s))$ are both decreasing functions of $s$. By the definition of $\bar{v}$, it holds that
\begin{equation*}
J(\sum_{i = 1}^p\alpha_i \delta_{a_i, \lambda_i} + \bar{v}) \leq J(\sum_{i = 1}^p\alpha_i \delta_{a_i, \lambda_i}) = J(h_0(\sum_{i = 1}^p\alpha_i \delta_{a_i, \lambda_i})).
\end{equation*}
By Lemma \ref{l:characterization}, the flow line $h_s(\sum_{i = 1}^p\alpha_i \delta_{a_i, \lambda_i})$ satisfies the (PS) condition if it does not go to infinity. Since the flow line started far away from these critical points at infinity and $d\lambda_{i_0}\leq C\lambda_{i_0}$, then it will take an infinite time for the flow line to go to infinity. Then the flow line would be down the level
$$
J(\sum_{i = 1}^p\alpha_i \delta_{a_i, \lambda_i} + \bar{v}) = J(\sum_{i = 1}^p\alpha_i \delta_{a_i, \lambda_i}) + o(1)
$$
before it exists from $V(p, \varepsilon)$. Thus there is at most one solution of the equation
\begin{equation}\label{e:flowequation}
J(h_s(\sum_{i = 1}^p\alpha_i \delta_{a_i, \lambda_i})) = J(\sum_{i = 1}^p\alpha_i \delta_{a_i, \lambda_i} + \bar{v}).
\end{equation}
Indeed, we assume $\sum_{i = 1}^p\alpha_i \delta_{a_i, \lambda_i}\in V(p, \varepsilon_1)$ with $\varepsilon_1 < \varepsilon/2$. Then the flow line will travel from $\partial V(p, \varepsilon_1)$ to $\partial V(p, \varepsilon)$. During this strip, it holds that
$$
- J'(h_s(\sum_{i = 1}^p\alpha_i \delta_{a_i, \lambda_i}))\geq C\left(\sum_{i = 1}^p\left(\frac{|\nabla K(a_i)|}{\lambda_i} + \frac{1}{\lambda_i^2}\right) + \sum_{i\neq j} \varepsilon_{ij}\right)\geq C(\varepsilon) > 0
$$
and
$$
d(\partial V(p, \varepsilon_1), \partial V(p, \varepsilon)) = a(\varepsilon),\quad |W_0|\leq C.
$$
Let $\Delta s$ denote the time to travel from $\partial V(p, \varepsilon_1)$ to $\partial V(p, \varepsilon)$, then we have $a(\varepsilon)\leq C \Delta s$. Let $\gamma(\varepsilon) = - \frac{C(\varepsilon)a(\varepsilon)}{C}$, then $J(h_s(\sum_{i = 1}^p\alpha_i \delta_{a_i, \lambda_i}))$ decreases at least $-\gamma(\varepsilon)$ during this strip. Hence, (\ref{e:flowequation}) has a unique solution. 

Conversely, starting from $\sum_{i = 1}^p\tilde{\alpha}_i \delta_{\tilde{a}_i, \tilde{\lambda}_i}$ in $V(p, \varepsilon_2)$, $\varepsilon_2$ sufficiently small, we consider the vector field $-W_0$. The flow line is $h_{-s}(\sum_{i = 1}^p\tilde{\alpha}_i \delta_{\tilde{a}_i, \tilde{\lambda}_i})$, we have to solve
\begin{equation}\label{e:flowequation3}
J(h_{-s}(\sum_{i = 1}^p\tilde{\alpha}_i \delta_{\tilde{a}_i, \tilde{\lambda}_i}) + \bar{v}(h_{-s}(\sum_{i = 1}^p\tilde{\alpha}_i \delta_{\tilde{a}_i, \tilde{\lambda}_i}))) = J(\sum_{i = 1}^p\tilde{\alpha}_i \delta_{\tilde{a}_i, \tilde{\lambda}_i}).
\end{equation}
It holds that
\begin{eqnarray*}
&&\frac{d}{ds}(J(h_{-s}(\sum_{i = 1}^p\tilde{\alpha}_i \delta_{\tilde{a}_i, \tilde{\lambda}_i}) + \bar{v}(h_{-s}(\sum_{i = 1}^p\tilde{\alpha}_i \delta_{\tilde{a}_i, \tilde{\lambda}_i}))))\\
&&=J'(h_{-s}(\sum_{i = 1}^p\tilde{\alpha}_i \delta_{\tilde{a}_i, \tilde{\lambda}_i}) + \bar{v}(h_{-s}(\sum_{i = 1}^p\tilde{\alpha}_i \delta_{\tilde{a}_i, \tilde{\lambda}_i})))(-W_0 - \frac{\partial \bar{v}}{\partial (\alpha_i, a_i, \lambda_i)}(W_0))\\
&&\geq C\left(\sum_{i = 1}^p\left(\frac{|\nabla K(\tilde{a}_i(s))|}{\tilde{\lambda}_i(s)} + \frac{1}{\tilde{\lambda}_i^2(s)}\right) + \sum_{i\neq j} \tilde{\varepsilon}_{ij}(s)\right) > 0.
\end{eqnarray*}
So $J(h_{-s}(\sum_{i = 1}^p\tilde{\alpha}_i \delta_{\tilde{a}_i, \tilde{\lambda}_i}) + \bar{v}(h_{-s}(\sum_{i = 1}^p\tilde{\alpha}_i \delta_{\tilde{a}_i, \tilde{\lambda}_i})))$ is an increasing function of $s$. Therefore there is at most one solution of (\ref{e:flowequation3}). We have here the same problem discussed in the first case. If (PS) is satisfied along the increasing flow lines of the vector field, then we can apply the same argument above to show that the flow line does not exit from $V(p, \varepsilon)$.
If (PS) is not satisfied, since $|d\lambda_{i_0}(W_0)|\leq C \lambda_{i_0}$, it would be in an infinite time, and during this time $[0, +\infty)$, we have
$$
\bar{u}_s = h_{-s}(\sum_{i = 1}^p\tilde{\alpha}_i \delta_{\tilde{a}_i, \tilde{\lambda}_i}) + \bar{v}(h_{-s}(\sum_{i = 1}^p\tilde{\alpha}_i \delta_{\tilde{a}_i, \tilde{\lambda}_i}))
$$
and
$$
\frac{d}{ds}J(\bar{u}_s)\geq C\left(\sum_{i = 1}^p\left(\frac{|\nabla K(\tilde{a}_i(s))|}{\tilde{\lambda}_i(s)} + \frac{1}{\tilde{\lambda}_i^2(s)}\right) + \sum_{i\neq j} \tilde{\varepsilon}_{ij}(s)\right).
$$
Since $\limsup_{s\to \infty}J(\bar{u}_s) = J(\bar{u}_\infty) < \infty$, we have
$$
\int_{0}^\infty\left(\sum_{i = 1}^p\left(\frac{|\nabla K(\tilde{a}_i(s))|}{\tilde{\lambda}_i(s)} + \frac{1}{\tilde{\lambda}_i^2(s)}\right) + \sum_{i\neq j} \tilde{\varepsilon}_{ij}(s)\right)ds < \infty.
$$
Therefore, up to a subsequence $s_k\to\infty$, it holds that 
$$
\frac{1}{\tilde{\lambda}_i^2(s_k)} + \sum_{i\neq j} \tilde{\varepsilon}_{ij}(s_k)\to 0
$$
and
$$
J(\bar{u}_{s_k}) - J(h_{-s_k}(\sum_{i = 1}^p\tilde{\alpha}_i \delta_{\tilde{a}_i, \tilde{\lambda}_i}))\to 0,
$$
since $\bar{v}(h_{-s_k}(\sum_{i = 1}^p\tilde{\alpha}_i \delta_{\tilde{a}_i, \tilde{\lambda}_i})\to 0$. So
\begin{eqnarray*}
\limsup_{s\to\infty}J(\bar{u}_s)& =& \limsup_{s\to\infty}J(h_{-s}(\sum_{i = 1}^p\tilde{\alpha}_i \delta_{\tilde{a}_i, \tilde{\lambda}_i})) > 
J(\sum_{i = 1}^p\tilde{\alpha}_i \delta_{\tilde{a}_i, \tilde{\lambda}_i}).
\end{eqnarray*}
By continuity of $J(\bar{u}_s)$, equation (\ref{e:flowequation3}) has a solution. 

Next, we will show that
\begin{equation}\label{e:equation1}
\sum_{i\neq j}\tilde{\varepsilon}_{ij} + \sum_{i}\frac{1}{\tilde{\lambda}_i^2} \to 0 \quad\text{as}\quad \sum_{i\neq j}\varepsilon_{ij} + \sum_{i}\frac{1}{\lambda_i^2} \to 0
\end{equation}
and conversely, and $d(a_i, \tilde{a}_i)\to 0$ under the above conditions. Set $\sum_{i = 1}^p\alpha_i\delta_{a_i(s), \lambda_i(s)} = h_s(\sum_{i = 1}^p\alpha_i\delta_{a_i, \lambda_i})$. Since $W_0$ has no action on the variables $\alpha_i$, we can write $W_0$ as follows,
$$
W_0 = \sum_{i = 1}^p\alpha_i\left(\frac{1}{\lambda_i(s)}\frac{\partial \delta_{a_i(s), \lambda_i(s)}}{\partial a_i(s)}\right)(\lambda_i(s)\dot{a}_i(s)) + \sum_{i = 1}^p\alpha_i\left(\lambda_i(s)\frac{\partial \delta_{a_i(s), \lambda_i(s)}}{\partial \lambda_i(s)}\right)\left(\frac{\dot{\lambda}_i(s)}{\lambda_i(s)}\right),
$$
where $\dot{a}_i(s)$ and $\dot{\lambda}_i(s)$ denote the action of $W_0$ on the variables $a_i$ and $\lambda_i$, that is to say, $\dot{a}_i = \frac{\partial a_i}{\partial W_0}$ and $\dot{\lambda}_i = \frac{\partial \lambda}{\partial W_0}$. $\frac{1}{\lambda_i(s)}\frac{\partial \delta_{a_i(s), \lambda_i(s)}}{\partial a_i(s)}$ and $\lambda_i(s)\frac{\partial \delta_{a_i(s), \lambda_i(s)}}{\partial \lambda_i(s)}$ are nearly orthogonal, both are $O(\delta_{a_i(s), \lambda_i(s)})$ and $0 < C_1\leq |\frac{1}{\lambda_i(s)}\frac{\partial \delta_{a_i(s), \lambda_i(s)}}{\partial a_i(s)}|\leq C$ and $0 < |\lambda_i(s)\frac{\partial \delta_{a_i(s), \lambda_i(s)}}{\partial \lambda_i(s)}| \leq C$. Since $W_0$ is bounded, $|\lambda_i(s)\dot{a}_i(s)| + |\frac{\dot{\lambda}_i(s)}{\lambda_i(s)}|\leq C'$, $i = 1,\cdot\cdot\cdot, p$. On the other hand, since $\varepsilon_{ij} = o(1)$, $\lambda_i(s)\frac{\partial \varepsilon_{ij}}{\partial \lambda_i(s)}$ and $\frac{1}{\lambda_i(s)}\frac{\partial \varepsilon_{ij}}{\partial a_i(s)}$ are both $O(\varepsilon_{ij})$ and
$$
|\frac{\partial}{\partial s}\varepsilon_{ij}| = |\lambda_i(s)\frac{\partial \varepsilon_{ij}}{\partial \lambda_i(s)}\frac{\dot{\lambda}_i(s)}{\lambda_i(s)} + \frac{1}{\lambda_i(s)}\frac{\partial \varepsilon_{ij}}{\partial a_i(s)}\lambda_i(s)\dot{a}_i(s)|\leq C\varepsilon_{ij},
$$
$$
e^{-cs}\leq \frac{\varepsilon_{ij}(s)}{\varepsilon_{ij}(0)}\leq e^{cs}\quad\text{ and }e^{-cs}\leq \frac{\lambda_i(s)}{\lambda_i(0)}\leq e^{cs}.
$$
Thus we have (\ref{e:equation1}).
On the other hand, we have $|\dot{a}_i(s)|\leq \frac{C}{\lambda_i(s)}\leq \frac{Ce^{cs}}{\lambda_i(0)}$, thus $|a_i(s) - a_i(0)|\leq Cs\frac{Ce^{cs}}{\lambda_i(0)}$. Since $s_0$ satisfies equation (\ref{e:flowequation}), it is bounded and so we have $d(a_i, \tilde{a}_i)\to 0$. This completes the proof of this lemma.
\end{proof}
{\bf Proof of Lemma \ref{l:morselemma}}. Lemma \ref{l:morselemma} follows from Lemma \ref{l:taylorformular} and \ref{l:lemma4.4}, since for any $\varepsilon_1 > 0$, there exist $\varepsilon > 0$ and $\varepsilon' > 0$ such that
$$
V(p, \varepsilon)\stackrel{h_s}{\longrightarrow}V(p, \varepsilon_1)\stackrel{h_{-s}}{\longrightarrow}V(p, \varepsilon')\supseteq V(p, \varepsilon).
$$
\qed

\section{Proof of Theorem \ref{t:main}.}
\begin{lemma}\label{l:criticalpointatinfinity}
Assume that (\ref{e:main}) has no solution. Then, the set of critical point at infinity of $J$ in $\Sigma^+$ is
$$
F^+_\infty = \{\phi_{\infty}^p = (y_{i_1},\cdot\cdot\cdot, y_{i_p})_\infty| (y_{i_1}, \cdot\cdot\cdot, y_{i_p})\in F^+\}.
$$
\end{lemma}
\begin{proof}
By Lemma \ref{l:characterization}, the Palais-Smale condition is satisfied along the flow lines of $W_0$ except in case (5). So we assume for a function $u = \sum_{i = 1}^j \alpha_i\delta_{a_j, \lambda_j} + v$ in $V(p, \varepsilon)$ that the concentration points $a_j$ converge to distinct critical points $y_{i_j}$ in $I^+$ and that all speeds of concentrations $\lambda_j$ are comparable. By the Morse lemma at infinity Lemma \ref{l:morselemma}, we have
$$
J(\sum_{i= 1}^p\alpha_i\delta_{a_i, \lambda_i} + v) = J(\sum_{i= 1}^p\alpha_i\delta_{\tilde{a}_i, \tilde{\lambda}_i}) + \frac{1}{2}J''(\sum_{i= 1}^p\alpha_i\delta_{a_i, \lambda_i})V_l\cdot V_l.
$$
The variable $V$ is completely independent from the others, and close to zero in a fixed Hilbert subspace. Minimizing with respect to $V$, the problem is reduced in a finite dimensional problem. In fact, we can do as $V$ is zero. Indeed, one can define on the $V$-variable the pseudogradient $\frac{\partial V}{\partial s} = -V$ and then $V(s) = e^{-s}V(0)$ will go to zero as $s$ goes to $+\infty$. From the proof of Lemma \ref{l:characterization}, we have $|\nabla K(\tilde{a}_j)| = o(1)$ and $\tilde{\varepsilon}_{ij} = o(\frac{1}{\tilde{\lambda}_i^2}) = o(\frac{1}{\tilde{\lambda}_j^2})$. Therefore, after a suitable change of variables, we have (drop the tilde for simplicity),
\begin{equation}\label{e:proof1}
J(u) = \frac{\sum_{i = 1}^p \alpha_i^2 S}{\left(\sum_{i = 1}^p \alpha_i^{\frac{2n}{n - 2\gamma}}K(a_i) S\right)^{\frac{n - 2\gamma}{n}}}\left[1 - \sum_{j= 1}^p\frac{\Delta K(y_{i_j})}{\lambda_j^2}\right].
\end{equation}
(\ref{e:proof1}) yields a split of the variables $a$ and $\lambda$. When $a = (a_1, \cdot\cdot\cdot, a_p)$ is equal to $y = (y_{i_1}, \cdot\cdot\cdot, y_{i_{p}})$, only the variables $\lambda_j$ can move. Since $-\Delta K(y_{i_j}) > 0$, in order to decrease $J(u)$, the only way is to increase the variables $\lambda_j$. When $\lambda_j = +\infty$ for $j = 1,\cdot\cdot\cdot,p$, one obtains the critical point at infinity $(y_{i_1},\cdot\cdot\cdot, y_{i_p})_{\infty}$. This completes the proof.
\end{proof}
\begin{lemma}\label{l:morseindex}
Let $\phi_{\infty}^p = (y_{i_1}, \cdot\cdot\cdot, y_{i_p})_\infty\in F_\infty^+$ and $i(\phi_\infty^{p})$ be the Morse index of $J$ at $\phi_\infty^p$. Then
$$
i(\phi_\infty^p) = p - 1 + \sum_{j = 1}^p(n - ind(K, y_{i_j})),
$$
where $ind(K, y_{i_j})$ is the Morse index of $K$ at $y_{i_j}$.
\end{lemma}
\begin{proof}
By (\ref{e:proof1}), the Morse index of $J$ at $\phi_{\infty}^p$ is equal to the Morse index of the functional
$$\gamma_0(\alpha, a) = \frac{\sum_{i = 1}^p \alpha_i^2 S}{\left(\sum_{i = 1}^p \alpha_i^{\frac{2n}{n - 2\gamma}}K(a_i) S\right)^{\frac{n - 2\gamma}{n}}}$$
at its critical point. Since $\gamma_0$ is homogeneous in the variable $\alpha$ and has a maximum point $(\frac{1}{K(a_1)^{\frac{n - 2\gamma}{4\gamma}}}, \frac{1}{K(a_2)^{\frac{n - 2\gamma}{4\gamma}}},\cdot\cdot\cdot, \frac{1}{K(a_n)^{\frac{n - 2\gamma}{4\gamma}}})$ with critical value $S^{\frac{2\gamma}{n}}\left(\sum_{i = 1}^p\frac{1}{K(a_i)^{\frac{n - 2\gamma}{2\gamma}}}\right)^{\frac{2\gamma}{n}}$. On the other hand, $\gamma_0$ has a single critical point $y = (y_{i_1}, y_{i_2}, \cdot\cdot\cdot, y_{i_p})$ in the $a$ variable. Thus, after a change of variables, we have have the following normal form,
\begin{equation}\label{e:equation2}
J(u) = S^{\frac{2\gamma}{n}}\left(\sum_{j = 1}^p\frac{1}{K(y_{i_j})^{\frac{n - 2\gamma}{2\gamma}}}\right)^{\frac{2\gamma}{n}}\left[1 - |\alpha^-|^2 + \sum_{j = 1}^p(|a_j^+|^2 - |a_j^-|^2)\right].
\end{equation}
Here $\alpha^-\in \mathbb{R}^{p - 1}$ is the coordinate of $\alpha$ and $a_j^+$ ($a_j^-$) is the coordinate of $a_j$ along the stable (unstable) manifold of $y_{i_j}$. This completes the proof.
\end{proof}
{\bf Proof of Theorem \ref{t:main}}. We follow the argument of \cite{YacoubANS2013}. Let $N$ be a submanifold of $\Sigma^+$ with dimension $k$, and let $\phi_\infty^p$ be a critical point at infinity with Morse index $i(\phi_\infty^p)\leq k$. We say that $\phi_\infty^p$ is dominated by $N$ and denote as $\phi_\infty^p < N$, if
$$
N\cap W_s(\phi_\infty^p) \neq \emptyset.
$$

Now, for any $k\in\mathbb{N}$ and any subset $X_k$ of $\{\phi_\infty^p\in F^+_\infty: i(\phi_\infty^p)\leq k\}$, we recall the following set given by
$$
X_k^\infty = \displaystyle\cup_{\phi_\infty^p\in X_k}W_u(\phi_\infty^p).
$$
$X^\infty_k$ is a stratified set of top dimension $k$. Without loss of generality, we assume it equal to $k$. Since $\Sigma^+$ is a contractible set, $X_k^\infty$ is contractible in $\Sigma^+$. More precisely, there exists a contraction $h: [0,1]\times X_k^\infty\to \Sigma^+$, i.e., $h$ is continuous and such that $\forall u\in X_k^\infty$, $h(0, u) = u$ and $h(1, u) = \tilde{u}$ a fixed point in $X_k^\infty$. Let
$$
\psi(X_k^\infty) = h([0, 1]\times X_k^\infty).
$$
Then $\psi(X_k^\infty)$ is contractible of dimension $k + 1$. Now, deform $\psi(X_k^\infty)$ by the flow lines of $-\partial J$. By transversality arguments, we can assume that the deformation avoids all critical points at infinity of Morse index greater then or equal to $k + 2$.

Using a result of Bahri and Rabinowitz (Proposition 7.24 and Theorem 8.2 of \cite{BahriandRabinowitzAIHN}), we have
\begin{eqnarray*}
\psi(X_k^\infty)&\simeq & \cup_{\phi_\infty^p < \psi(X_k^\infty), ind(\phi_\infty^p)\leq k + 1}W_u(\phi_\infty^p)\\
&\simeq & X_k^\infty \cup \cup_{\phi_\infty^p\in F^p_\infty\setminus X_k, \phi_\infty^p < \psi(X_k^\infty), ind(\phi_\infty^p)\leq k + 1}W_u(\phi_\infty^p).
\end{eqnarray*}
Now, taking $k = k_0$, where $k_0$ is the integer in Theorem \ref{t:main}, and by condition (H2), we have
$$
\psi(X_{k_0}^\infty) \simeq X_{k_0}^\infty.
$$

Now $X_{k_0}^\infty$ is a finite CW complex in dimension $k_0$, and the $j$-dimensional cells of $X_{k_0}^\infty$ are the unstable manifolds of critical points at infinity of Morse index $ind(\phi_\infty^p) = j$.
Hence we have
$$
1 = \chi(\psi(X_{k_0}^p)) = \chi(X_{k_0}^\infty) = \sum_{\phi_\infty^p\in X_{k_0}}(-1)^{i(\phi_\infty^p)},
$$
which contradicts the assumption (H1). Hence, there is a solution $w$. And it is easy to derive from the above arguments that its Morse index $ind(w)\leq k_0 + 1$. This completes the proof. \qed 
\section{Proof of Theorem \ref{t:main1}.}
Since $K = 1 + \varepsilon K(x)$, the functional is
\begin{equation*}
J(u) = \frac{\|u\|^2}{\left(\int_{\mathbb{S}^{n}}(1 + \varepsilon K_0) u^{\frac{2n}{n - 2\gamma}}dvol_{g_0}\right)^{\frac{n -2\gamma}{n}}}.
\end{equation*}
If $\varepsilon = 0$, we obtain the Yamabe functional
\begin{equation*}
J_0(u) = \frac{\|u\|^2}{\left(\int_{\mathbb{S}^{n}} u^{\frac{2n}{n - 2\gamma}}dvol_{g_0}\right)^{\frac{n -2\gamma}{n}}}.
\end{equation*}

Let $\sigma$ be the minimum of $J_0$ on $\Sigma^+$, then
$$
\sigma = J_0(\delta_{a, \lambda}) = S^{\frac{2\gamma}{n}},
$$
where $S$ is the constant defined in \ref{l:sharpconstants}. Since
$$
J(u) = J_0(u)\frac{1}{\left(1 + \varepsilon(\int_{\mathbb{S}^{n}} u^{\frac{2n}{n - 2\gamma}}dvol_{g_0})^{-1}\int_{\mathbb{S}^{n}} K u^{\frac{2n}{n - 2\gamma}}dvol_{g_0}\right)^{\frac{n -2\gamma}{n}}},
$$
$$
J(u) = J_0(u)(1 + O(\varepsilon)),
$$
with $O(\varepsilon)$ is independent of $u$, when $|\varepsilon|$ is small enough,
\begin{equation}\label{e:equivalence}
J^{\sigma + \eta}\subset J_0^{\sigma + 2\eta}\subset J^{\sigma + 3\eta}.
\end{equation}
We recall that $J^\beta = \{u\in \Sigma^+: J(u)\leq \beta\}$, $\beta\in \mathbb{R}$.

By (\ref{e:equation2}), the critical level corresponding to a critical point at infinity with $p$ bubbles is
$$
J((y_1,\cdot\cdot\cdot, y_p)_\infty) = S^{\frac{2\gamma}{n}}\left(\sum_{j = 1}^p\frac{1}{K(y_{j})^{\frac{n - 2\gamma}{2\gamma}}}\right)^{\frac{2\gamma}{n}},
$$
which tends to $S^{\frac{2\gamma}{n}}p^{\frac{2\gamma}{n}}$ in our current case as $\varepsilon\to 0$. Let $\eta = \frac{\sigma}{4}$, when $|\varepsilon|$ is sufficiently small, then critical points at infinity made of two bubbles or more are above the level $\sigma + 3\sigma$ and the ones with a single bubble are below the level $\sigma + \eta$. Therefore, $J^{\sigma + 3\eta}\simeq J^{\sigma + \eta}$. By (\ref{e:equivalence}), $J^{\sigma + 2\eta}_0\simeq J^{\sigma + \eta}$. Since $J^{\sigma + 2\eta}_0$ is contractible, $J^{\sigma + \eta}$ is also contractible.

Let $k_0$ be the integer which achieves the maximum of (\ref{e:equation3}). Then
\begin{equation}\label{e:equation4}
\sum_{y\in I^+, n - ind(K, y)\leq k}(-1)^{n - ind(K, y)}\neq 1.
\end{equation}
Let
$$
X_{k_0} = \{(y)_\infty|y\in I^+, i(y)_\infty\leq k_0\},
$$
which is the set of critical points at infinity with a single bubble and having Morse index $\leq k_0$. By Lemma \ref{l:morseindex}, (\ref{e:equation4}) becomes
\begin{equation*}
\sum_{(y)_\infty\in X_{k_0}}(-1)^{n - ind(K, y)}\neq 1,
\end{equation*}
which is condition (H1) in Theorem \ref{t:main}.

Let
$$
X_{k_0}^\infty = \cup_{(y)_\infty\in X_{k_0}}W_u(y)_\infty.
$$
Then $X_{k_0}^\infty\subset J^{\sigma + \eta}$. Since $J^{\sigma + \eta}$ is contractible, $X_{k_0}^\infty$ is contractible in $J^{\sigma + \eta}$. More precisely, there exists a contraction $h: [0,1]\times X_{k_0}^\infty\to J^{\sigma + \eta}$, i.e., $h$ is continuous and such that $\forall u\in X_{k_0}^\infty$, $h(0, u) = u$ and $h(1, u) = \tilde{u}$ a fixed point in $X_{k_0}^\infty$. Set
$$
\psi(X_{k_0}^\infty) = h([0, 1]\times X_{k_0}^\infty).
$$

Now let $\phi_\infty^p\in F^+_\infty\setminus X_{k_0}$ such that $i(\phi_\infty^p)\leq k_0 + 1$. If $\phi_\infty^p$ has at least two bubbles, then $\psi(X_{k_0}^\infty)\cap W_s(\delta_\infty^p) = \emptyset$ since $\psi(X_{k_0}^\infty)\subset J^{\sigma + \eta}$. So it remains to consider the critical points at infinity with a single bubble $(y)_\infty\in X_{k_0}$ such that $i(y)_\infty = k_0 + 1$. Since $k_0\in \mathbb{T}$, such critical points at infinity does not exist. Thus condition (H2) in Theorem \ref{t:main} is satisfied. By Theorem \ref{t:main}, we have the desired result. This completes the proof.\qed 
\newpage
\begin{appendix}
\section{Proof of Lemma \ref{l:expansion}.}
We write
\begin{equation*}
J(u) = \frac{\|u\|^2}{\left(\int_{\mathbb{S}^{n}} K u^{\frac{2n}{n - 2\gamma}}dvol_{g_0}\right)^{\frac{n -2\gamma}{n}}} = \frac{N}{D}.
\end{equation*}
To prove Lemma \ref{l:expansion}, we only need to give the expansions of $N$ and $D$.
First, we expand the numerator $N$ as follows
\begin{eqnarray*}
N = \|u\|^2 = \|\sum_{i = 1}^p\alpha_i\delta_{a_i, \lambda_i} + v\|^2 = \sum_{i = 1}^p \alpha_i^2 \|\delta_i\|^2 + \sum_{i\neq j} \alpha_i\alpha_j \langle\delta_i, \delta_j\rangle + \|v\|^2.
\end{eqnarray*}
Here, the other terms are zero since $v$ satisfies condition (\ref{e:orthonormal}). 
\begin{lemma}\label{l:sharpconstants}
It holds that
$$
\int_{\mathbb{S}^{n}}\delta_{a_i, \lambda_i}P_\gamma \delta_{a_i, \lambda_i} dvol_{g_0} = S,
$$
where $S$ is the sharp constant for the following conformally invariant Sobolev inequality
\begin{equation*}
\|f\|^2_{L^{2^*}(\mathbb{S}^{n})}\leq S\int_{\mathbb{S}^{n}}fP_\gamma f dvol_{g_0}, \quad 2^* = \frac{2n}{n - 2\gamma}.
\end{equation*}
\end{lemma}
Lemma \ref{l:sharpconstants} was proved by Lieb in \cite{LiebSCHLSAnnMath1983}.

\begin{lemma}\label{l:A2}
For $i\neq j$, we have
$$
\int_{\mathbb{S}^{n}}\delta_{a_j, \lambda_j}P_\gamma \delta_{a_i, \lambda_i} dvol_{g_0} = c_0^{\frac{2n}{n-2\gamma}}c_1\omega_n\varepsilon_{ij}\left(1 + o(1)\right),\quad\text{with  } c_1 = \int_{0}^{\infty}\frac{r^{n - 1}}{(1 + r^2)^{\frac{n + 2\gamma}{2}}}dr.
$$
\end{lemma}
\begin{proof}
\begin{eqnarray*}
I &=& \int_{\mathbb{S}^{n}}\delta_{a_j, \lambda_j}P_\gamma \delta_{a_i, \lambda_i} dvol_{g_0} = c_0^{\frac{2n}{n-2\gamma}}\int_{\mathbb{R}^n}\frac{\lambda_i^{\frac{n +2 \gamma}{2}}}{|1 + \lambda_i^2|\zeta - g_i|^2|^{\frac{n +2 \gamma}{2}}}\frac{\lambda_j^{\frac{n-2\gamma}{2}}}{|1 + \lambda_j^2|\zeta - g_j|^2|^{\frac{n-2\gamma}{2}}}d\zeta\\
& = & c_0^{\frac{2n}{n-2\gamma}}\left(\frac{\lambda_j}{\lambda_i}\right)^{^{\frac{n - 2\gamma}{2}}}\int_{\mathbb{R}^n}\frac{1}{|1 + |\zeta'|^2|^{\frac{n +2 \gamma}{2}}}\frac{1}{|1 + |\frac{\lambda_j}{\lambda_i}\zeta' + \lambda_j d_{ij}|^2|^{\frac{n - 2\gamma}{2}}}d\zeta'.
\end{eqnarray*}
Let
$$
\mu = \max\left(\frac{\lambda_i}{\lambda_j}, \frac{\lambda_j}{\lambda_i}, \lambda_i\lambda_jd_{ij}^2\right),\quad \text{where } d_{ij} = d(g_i, g_j).
$$
First, we assume $\mu = \frac{\lambda_i}{\lambda_j}$. Note that
$$
1 + |\frac{\lambda_j}{\lambda_i}\zeta' + \lambda_j d_{ij}|^2 = \left(1 + \lambda_j^2|d_{ij}|^2\right)\left(1 + 2\frac{\lambda_j}{1 + \lambda_j^2|d_{ij}|^2}d_{ij}\frac{\lambda_j}{\lambda_i}\zeta' +\left(\frac{\lambda_j}{\lambda_i}\right)^2 \frac{|\zeta'|^2}{1 + \lambda_j^2|d_{ij}|^2}\right).
$$
If $|\zeta'|\leq \frac{1}{4}\frac{\lambda_i}{\lambda_j}$, then
\begin{eqnarray*}
\left(1 + |\frac{\lambda_j}{\lambda_i}\zeta' + \lambda_j d_{ij}|^2\right)^{-\frac{n - 2\gamma}{2}} &=& \left(1 + \lambda_j^2|d_{ij}|^2\right)^{-\frac{n - 2\gamma}{2}}\\
&&\left(1 - (n - 2\gamma)\frac{\lambda_j}{1 + \lambda_j^2|d_{ij}|^2}d_{ij}\frac{\lambda_j}{\lambda_i}\zeta' + O\left(\frac{\lambda_j}{\lambda_i}\right)^2 |\zeta'|^2\right).
\end{eqnarray*}
Thus we have
$$
\int_{B(0, \frac{1}{4}\frac{\lambda_i}{\lambda_j})}\frac{1}{|1 + |\zeta'|^2|^{\frac{n + 2\gamma}{2}}}d\zeta' = \omega_n\int_{0}^{\infty}\frac{r^{n-1}}{(1 + r^2)^{\frac{n + 2\gamma}{2}}}dr + O\left(\left(\frac{\lambda_j}{\lambda_i}\right)^{2\gamma}\right),
$$
$$
\int_{B(0, \frac{1}{4}\frac{\lambda_i}{\lambda_j})^c}\frac{1}{|1 + |\zeta'|^2|^{\frac{n + 2\gamma}{2}}}d\zeta' = O\left(\left(\frac{\lambda_j}{\lambda_i}\right)^{2\gamma}\right),
$$
and
$$
\int_{B(0, \frac{1}{4}\frac{\lambda_i}{\lambda_j})}\frac{|\zeta'|^2}{|1 + |\zeta'|^2|^{\frac{n + 2\gamma}{2}}}d\zeta' = O\left(\left(\frac{\lambda_i}{\lambda_j}\right)^{2 - 2\gamma}\right).
$$
Hence
$$
I = c_0^{\frac{2n}{n-2\gamma}}\left(\frac{\lambda_j}{\lambda_i}\right)^{^{\frac{n - 2\gamma}{2}}}\left(1 + \lambda_j^2|d_{ij}|^2\right)^{-\frac{n -2 \gamma}{2}}\left[\omega_n\int_{0}^{\infty}\frac{r^{n-1}}{(1 + r^2)^{\frac{n + 2\gamma}{2}}}dr + O\left(\left(\frac{\lambda_j}{\lambda_i}\right)^{2\gamma}\right)\right].
$$
Under the assumption $\mu = \frac{\lambda_i}{\lambda_j}$, we have, as $\varepsilon_{ij}\to 0$,
$$I = c_0^{\frac{2n}{n-2\gamma}}\varepsilon_{ij}\left[\omega_n\int_{0}^{\infty}\frac{r^{n-1}}{(1 + r^2)^{\frac{n + 2\gamma}{2}}}dr + O\left(\varepsilon_{ij}^{\frac{4\gamma}{n - 2\gamma}}\right)\right].$$

Similarly, the result of the lemma holds in the case $\mu = \frac{\lambda_j}{\lambda_i}$.

Finally, we consider the case $\mu = \lambda_i\lambda_j|d_{ij}|^2$.
We rewrite $I$ as
\begin{eqnarray*}
I  =  c_0^{\frac{2n}{n - 2\gamma}}\int_{\mathbb{R}^n}\frac{1}{|1 + |\zeta'|^2|^{\frac{n+2\gamma}{2}}}\frac{1}{|\frac{\lambda_i}{\lambda_j} + |\sqrt{\frac{\lambda_j}{\lambda_i}}\zeta' + \sqrt{\lambda_i\lambda_j} d_{ij}|^2|^{\frac{n-2\gamma}{2}}}d\zeta'.
\end{eqnarray*}
Without loss of generality, we assume $\lambda_i\leq \lambda_j$, then we have
$$
\frac{\lambda_i}{\lambda_j} + |\sqrt{\frac{\lambda_j}{\lambda_i}}\zeta' + \sqrt{\lambda_i\lambda_j} d_{ij}|^2 = \left(\frac{\lambda_i}{\lambda_j} + \lambda_i\lambda_j d_{ij}^2\right)\left(1 + \frac{\frac{\lambda_j}{\lambda_i}|\zeta'|^2 + 2\lambda_j\zeta' d_{ij}}{\frac{\lambda_i}{\lambda_j} + \lambda_i\lambda_j d_{ij}^2}\right).
$$
By the same arguments used in the first case we have
\begin{eqnarray*}
&&\int_{|\zeta'|\leq \frac{\sqrt{\mu}}{10}}\frac{1}{|1 + |\zeta'|^2|^{\frac{n+2\gamma}{2}}}\frac{1}{|\frac{\lambda_i}{\lambda_j} + |\sqrt{\frac{\lambda_j}{\lambda_i}}\zeta' + \sqrt{\lambda_i\lambda_j} d_{ij}|^2|^{\frac{n-2\gamma}{2}}}d\zeta'\\ &&= \varepsilon_{ij}\left[\omega_n\int_{0}^{\infty}\frac{r^{n - 1}}{(1 + r^2)^{\frac{n + 2\gamma}{2}}}dr + O\left(\varepsilon_{ij}^{\frac{4\gamma}{n -2 \gamma}}\right)\right].
\end{eqnarray*}
Set
$$
B_1: = \left\{\zeta'\in \mathbb{R}^n|  |\zeta' + \lambda_i d_{ij}|\leq \frac{1}{10}\lambda_i|d_{ij}|\right\},\quad B_2:=\left\{\zeta'\in \mathbb{R}^n|  |\zeta'|\leq \frac{\sqrt{\mu}}{10}\right\}.
$$
Then we have
$$
\int_{(B_1 \cup B_2)^c}\frac{1}{|1 + |\zeta'|^2|^{\frac{n+2\gamma}{2}}}\frac{1}{|\frac{\lambda_i}{\lambda_j} + |\sqrt{\frac{\lambda_j}{\lambda_i}}\zeta' + \sqrt{\lambda_i\lambda_j} d_{ij}|^2|^{\frac{n-2\gamma}{2}}}d\zeta'\leq \frac{c}{\mu^{\frac{n-2\gamma}{2}}}\int_{\sqrt{\mu}}^\infty \frac{r^{n - 1}}{(1 + r^2)^{\frac{n+2\gamma}{2}}}dr.
$$
Therefore,
$$
\int_{(B_1\cup B_2)^c}\frac{1}{|1 + |\zeta'|^2|^{\frac{n+2\gamma}{2}}}\frac{1}{|\frac{\lambda_i}{\lambda_j} + |\sqrt{\frac{\lambda_j}{\lambda_i}}\zeta' + \sqrt{\lambda_i\lambda_j} d_{ij}|^2|^{\frac{n-2\gamma}{2}}}d\zeta' = O(\frac{1}{\mu^2}) = O(\varepsilon_{ij}^{\frac{4}{n - 2\gamma}}).
$$
On $B_1$, we have $|\zeta'|\geq \frac{9}{10}\lambda_i|d_{ij}|$, this leads to
$$
\int_{B_1}\frac{1}{|1 + |\zeta'|^2|^{\frac{n+2\gamma}{2}}}\frac{1}{|\frac{\lambda_i}{\lambda_j} + |\sqrt{\frac{\lambda_j}{\lambda_i}}\zeta' + \sqrt{\lambda_i\lambda_j} d_{ij}|^2|^{\frac{n-2\gamma}{2}}}d\zeta' = O(\frac{1}{\mu^4}) = O(\varepsilon_{ij}^{\frac{8}{n - 2\gamma}}).
$$
This completes the proof Lemma \ref{l:A2}.
\end{proof}
Now let us consider the denominator $D$ of $J$,
$$
D^{\frac{n}{n - 2\gamma}} = \int_{\mathbb{S}^n} K \left(\sum_{i = 1}^p\alpha_i\delta_{a_i, \lambda_i} + v\right)^{\frac{2n}{n - 2\gamma}} dvol_{g_0}.
$$
First, since
\begin{eqnarray*}
&&\left|\int_{\mathbb{S}^n} K \left(\sum_{i = 1}^p\alpha_i\delta_{a_i, \lambda_i} + v\right)^{\frac{2n}{n - 2\gamma}}dvol_{g_0} - \int_{\mathbb{S}^n} K \left(\sum_{i = 1}^p\alpha_i\delta_{a_i, \lambda_i}\right)^{\frac{2n}{n - 2\gamma}}dvol_{g_0}\right.\\
&&\quad\quad\quad\quad\quad\quad\quad\quad\quad\quad\quad\quad\quad\quad\quad\quad - \frac{2n}{n - 2\gamma}\int_{\mathbb{S}^n} K \left(\sum_{i = 1}^p\alpha_i\delta_{a_i, \lambda_i}\right)^{\frac{n + 2\gamma}{n - 2\gamma}}v\,\,dvol_{g_0}\\
&&\left.\quad\quad\quad\quad\quad\quad\quad\quad\quad\quad\quad\quad\quad\quad- \frac{2n^2 + 4n\gamma}{n^2 - 4\gamma^2}\int_{\mathbb{S}^n} K \left(\sum_{i = 1}^p\alpha_i\delta_{a_i, \lambda_i}\right)^{\frac{4\gamma}{n - 2\gamma}}v^2\,\,dvol_{g_0}\right|\\
&&\leq C\left[\int_{\mathbb{S}^n} v^{\frac{2n}{n - 2\gamma}}\,\,dvol_{g_0} + \int_{\mathbb{S}^n}\left(\sum_{i = 1}^p\alpha_i\delta_{a_i, \lambda_i}\right)^{\frac{6\gamma - n}{n -2\gamma}}\inf\left((\sum_{i = 1}^p\alpha_i\delta_{a_i, \lambda_i})^3, |v|^3\right)\,\,dvol_{g_0}\right]\\
&& \leq C\left[\int_{\mathbb{S}^n} v^{\frac{2n}{n - 2\gamma}}\,\,dvol_{g_0} + \int_{\{x|\sum_{i = 1}^p\alpha_i\delta_{a_i, \lambda_i}\geq |v(x)|\}}\left(\sum_{i = 1}^p\alpha_i\delta_{a_i, \lambda_i}\right)^{\frac{6\gamma - n}{n - 2\gamma}}|v|^3\,\,dvol_{g_0}\right]\\
&& \leq C\int_{\mathbb{S}^n} v^{\frac{2n}{n - 2\gamma}}\,\,dvol_{g_0}\leq C\left(\int_{\mathbb{S}^n}vP_\gamma v\,\,dvol_{g_0}\right)^{\frac{n}{n - 2\gamma}},
\end{eqnarray*}
we have
\begin{eqnarray*}
&&\int_{\mathbb{S}^{n}} K \left(\sum_{i = 1}^p\alpha_i\delta_{a_i, \lambda_i} + v\right)^{\frac{2n}{n - 2\gamma}}dvol_{g_0} = \int_{\mathbb{S}^{n}} K \left(\sum_{i = 1}^p\alpha_i\delta_{a_i, \lambda_i}\right)^{\frac{2n}{n - 2\gamma}}dvol_{g_0}\\
&&\quad\quad\quad\quad\quad\quad\quad\quad\quad\quad\quad\quad\quad\quad\quad\quad + \frac{2n}{n - 2\gamma}\int_{\mathbb{S}^{n}} K \left(\sum_{i = 1}^p\alpha_i\delta_{a_i, \lambda_i}\right)^{\frac{n + 2\gamma}{n - 2\gamma}}v\,\,dvol_{g_0}\\ & & \quad\quad\quad\quad\quad + \frac{2n^2 + 4n\gamma}{ n^2 - 4\gamma^2}\int_{\mathbb{S}^{n}} K \left(\sum_{i = 1}^p\alpha_i\delta_{a_i, \lambda_i}\right)^{\frac{4\gamma}{n - 2 \gamma}}v^2\,\,dvol_{g_0}+ O\left(\int_{\mathbb{S}^n}vP_\gamma v\,\,dvol_{g_0}\right)^{\frac{n}{n - 2\gamma}}.
\end{eqnarray*}
\begin{lemma}
\begin{eqnarray*}
&&\int_{\mathbb{S}^{n}} K \left(\sum_{i = 1}^p\alpha_i\delta_{a_i, \lambda_i}\right)^{\frac{2n}{n - 2\gamma}}dvol_{g_0} =
\sum_{i = 1}^p\alpha_i^{\frac{2n}{n - 2\gamma}}K(a_i)S + \sum_{i = 1}^p\alpha_i^{\frac{2n}{n - 2\gamma}}c_2\frac{\Delta K(x_i)}{\lambda_i^2}\\ &&+ o\left(\sum_{i = 1}^p\frac{1}{\lambda_i^{2}}\right)+ \frac{2n}{n - 2\gamma}\sum_{i \neq j}\alpha_i^{\frac{n + 2\gamma}{n - 2\gamma}}\alpha_j K(a_i)c_0^{\frac{2n}{n-2\gamma}}c_1\omega_n\varepsilon_{ij}+ o\left(\sum_{i \neq j}\varepsilon_{ij}\right)\\ &&+ O\left(\sum_{i \neq j}\frac{1}{\lambda_i^{\frac{n + 2 \gamma}{2}}\lambda_j^{\frac{n - 2\gamma}{2}}}\right)+ O\left(\sum_{i\neq j}\varepsilon_{ij}^{\frac{n}{n - 2\gamma}}\log \varepsilon_{ij}^{-1}\right), \quad c_2 = c_0^{\frac{2n}{n - 2\gamma}}\int_{\mathbb{R}^n}|\zeta|^2w_{0, 1}^{\frac{2n}{n - 2\gamma}}d\zeta.
\end{eqnarray*}
\end{lemma}
\begin{proof}
\begin{eqnarray*}
&&\int_{\mathbb{S}^{n}} K \left(\sum_{i = 1}^p\alpha_i\delta_{a_i, \lambda_i}\right)^{\frac{2n}{n - 2\gamma}}dvol_{g_0} = \sum_{i = 1}^p\alpha_i^{\frac{2n}{n - 2\gamma}}\int_{\mathbb{S}^{n}} K \delta_{a_i, \lambda_i}^{\frac{2n}{n - 2\gamma}}dvol_{g_0}+\\ && \frac{2n}{n - 2\gamma}\sum_{i \neq j}\alpha_i^{\frac{n + 2\gamma}{n - 2\gamma}}\alpha_j\int_{\mathbb{S}^{n}} K \delta_{a_i, \lambda_i}^{\frac{n + 2\gamma}{n - 2\gamma}}\delta_{a_j, \lambda_j}dvol_{g_0} + O\left(\sum_{i \neq j}\int_{\mathbb{S}^{n}} \delta_{a_i, \lambda_i}^{\frac{4\gamma}{n - 2\gamma}}\inf(\delta_{a_j, \lambda_j}, \delta_{a_i, \lambda_i})^2dvol_{g_0}\right).
\end{eqnarray*}
And it holds that
$$
\sum_{i \neq j}\int_{\mathbb{S}^{n}} \delta_{a_i, \lambda_i}^{\frac{4\gamma}{n - 2\gamma}}\inf(\delta_{a_j, \lambda_j}, \delta_{a_i, \lambda_i})^2dvol_{g_0}\leq C \sum_{i \neq j} \varepsilon_{ij}^{\frac{n}{n - 2\gamma}}\log \varepsilon_{ij}^{-1}.
$$
Since
\begin{eqnarray*}
\int_{\mathbb{S}^{n}} K \delta_{a_i, \lambda_i}^{\frac{2n}{n - 2\gamma}}dvol_{g_0} &=& \int_{\mathbb{R}^n} \tilde{K}(x') w_{g_i, \lambda_i}^{\frac{2n}{n - 2\gamma}}\theta_0dx'\\
&=& \int_{\mathbb{R}^n} \left(\tilde{K}(x') - \tilde{K}(g_i)\right)w_{g_i, \lambda_i}^{\frac{2n}{n - 2\gamma}}dx' +  \tilde{K}(g_i)\int_{\mathbb{R}^n} w_{g_i, \lambda_i}^{\frac{2n}{n - 2\gamma}}dx'
\end{eqnarray*}
\begin{eqnarray*}
&= &\int_{B(g_i,\rho)} \left(\tilde{K}(x') - \tilde{K}(g_i)\right)w_{g_i, \lambda_i}^{\frac{2n}{n - 2\gamma}}dx'\\ &&+ \int_{B(g_i,\rho)^c} \left(\tilde{K}(x') - \tilde{K}(g_i)\right)w_{g_i, \lambda_i}^{\frac{2n}{n - 2\gamma}}dx'\\ &&+  \tilde{K}(g_i)\int_{\mathbb{R}^n} w_{g_i, \lambda_i}^{\frac{2n}{n - 2\gamma}}dx'\\
&=& \int_{B(0,\rho')} \left(\tilde{K}(x') - \tilde{K}(0)\right)w_{0, \lambda_i}^{\frac{2n}{n - 2\gamma}}dx'\\ &&+ \int_{B(0,\rho')^c} \left(\tilde{K}(x') - \tilde{K}(0)\right)w_{0, \lambda_i}^{\frac{2n}{n - 2\gamma}}dx'\\ &&+  K(a_i)\int_{\mathbb{S}^3} \delta_{a_i, \lambda_i}^{\frac{2n}{n - 2\gamma}}dvol_{g_0}
\end{eqnarray*}
$$
\int_{B(0,\rho')^c} \left(\tilde{K}(x') - \tilde{K}(0)\right)w_{0, \lambda_i}^{\frac{2n}{n - 2\gamma}}dx'\leq \frac{C}{\lambda_i^{n}},
$$
and
$$
\int_{B(0,\rho')} \left(\tilde{K}(x') - \tilde{K}(0)\right)w_{0, \lambda_i}^{\frac{2n}{n - 2\gamma}}dx' = c_0^{\frac{2n}{n - 2\gamma}}\int_{\mathbb{R}^n}|\zeta|^2w_{0, 1}^{\frac{2n}{n - 2\gamma}}dx'\frac{\Delta K(x_i)}{\lambda_i^2} + O\left(\frac{1}{\lambda_i^{n}}\right),
$$
we have
$$
\int_{\mathbb{S}^{n}} K \delta_{a_i, \lambda_i}^{\frac{2n}{n - 2\gamma}}dvol_{g_0} = K(a_i)S + c_0^{\frac{2n}{n - 2\gamma}}\int_{\mathbb{R}^n}|\zeta|^2w_{0, 1}^{\frac{2n}{n - 2\gamma}}dx'\frac{\Delta K(x_i)}{\lambda_i^2} + o\left(\frac{1}{\lambda_i^{2}}\right).
$$
Finally, we expand the term $\int_{\mathbb{S}^{n}} K(x) \delta_{a_i, \lambda_i}^{\frac{n + 2\gamma}{n - 2\gamma}}\delta_{a_j, \lambda_j}dvol_{g_0}$.
\begin{eqnarray*}
\int_{\mathbb{S}^{n}} K(x) \delta_{a_i, \lambda_i}^{\frac{n + 2\gamma}{n - 2\gamma}}\delta_{a_j, \lambda_j}dvol_{g_0} &=& \int_{B(g_i, \rho)} \left(\tilde{K}(x') - \tilde{K}(g_i)\right) w_{g_i, \lambda_i}^{\frac{n + 2\gamma}{n - 2\gamma}}w_{g_j, \lambda_j}dx'\\
&& + \int_{B(g_i, \rho)^c} \left(\tilde{K}(x') - \tilde{K}(g_i)\right) w_{g_i, \lambda_i}^{\frac{n + 2\gamma}{n - 2\gamma}}w_{g_j, \lambda_j}dx'\\
&& + \tilde{K}(g_i)\int_{\mathbb{R}^n}w_{g_i, \lambda_i}^{\frac{n + 2\gamma}{n - 2\gamma}}w_{g_j, \lambda_j}dx'.
\end{eqnarray*}
And we have
\begin{eqnarray*}
&&\sum_{i\neq j}\int_{B(g_i, \rho)} \left(\tilde{K}(x') - \tilde{K}(g_i)\right) w_{g_i, \lambda_i}^{\frac{n + 2\gamma}{n - 2\gamma}}w_{g_j, \lambda_j}dx'\\
&& = O\left(\sum_{i\neq j}\|\nabla \tilde{K}(g_i)\|\int_{B(g_i, \rho)} \|x'\| w_{g_i, \lambda_i}^{\frac{n + 2\gamma}{n - 2\gamma}}w_{g_j, \lambda_j}dx'\right)\\
&& = O\left(\sum_{i\neq j}\left(\int_{B(g_i, \rho)} \|x'\|^{\frac{n}{2\gamma}} w_{g_i, \lambda_i}^{\frac{n}{n - 2\gamma}}dx'\right)^{\frac{2\gamma}{n}}\left(\int_{B(g_i, \rho)} w_{g_i, \lambda_i}^{\frac{n}{n - 2\gamma}}w_{g_j, \lambda_j}^{\frac{n}{n -2\gamma}}dx'\right)^{\frac{n - 2\gamma}{n}}\right)\\
&& = O\left(\sum_{i\neq j}\left(\frac{1}{\lambda_i}\right)\left(\varepsilon_{ij}^{\frac{n}{n - 2\gamma}}\log \varepsilon_{ij}^{-1}\right)^{\frac{n-2\gamma}{n}}\right)\\
&& = O\left(\sum_{i = 1}^p\left(\frac{1}{\lambda_i}\right)^{\frac{n}{2\gamma}} + \sum_{i\neq j}\varepsilon_{ij}^{\frac{n}{n- 2\gamma}}\log \varepsilon_{ij}^{-1}\right).
\end{eqnarray*}
Hence
\begin{eqnarray*}
\int_{\mathbb{S}^{n}} K(x) \delta_{a_i, \lambda_i}^{\frac{n + 2\gamma}{n - 2\gamma}}\delta_{a_j, \lambda_j}dvol_{g_0} &=& c_0^{\frac{4}{2-\gamma}}K(a_i)c_1\omega_3\varepsilon_{ij}\left(1 + o(1)\right)\\
&+& O\left(\sum_{i = 1}^p\left(\frac{1}{\lambda_i}\right)^{\frac{n}{2\gamma}} +  \sum_{i \neq j}\frac{1}{\lambda_i^{\frac{n + 2\gamma}{2}}\lambda_j^{\frac{n - 2\gamma}{2}}} + \sum_{i\neq j}\varepsilon_{ij}^{\frac{n}{n- 2\gamma}}\log \varepsilon_{ij}^{-1}\right).
\end{eqnarray*}
Combine the above estimates, we have the assertion of this lemma.
\end{proof}
\begin{lemma}
If $0 < \gamma \leq \frac{n}{6}$, then
\begin{eqnarray*}
&&\int_{\mathbb{S}^{n}} K \left(\sum_{i = 1}^p\alpha_i\delta_{a_i, \lambda_i}\right)^{\frac{n + 2\gamma}{n - 2\gamma}}v\,\,dvol_{g_0}\\ && = O\left(\left(\int_{\mathbb{S}^n}vP_\gamma v\right)^{\frac{1}{2}}\right)\left(O\left(\sum_{i\neq j} \left(\varepsilon_{ij}\right)^{\frac{n + 2\gamma}{2(n - 2\gamma)}} \left(\log \varepsilon_{ij}^{-1}\right)^{\frac{n + 2\gamma}{2n}}\right)+ \sum_{i = 1}^p\left(\frac{|\nabla K(a_i)|}{\lambda_i}+\frac{1}{\lambda_i^2}\right)\right).
\end{eqnarray*}
If $1 > \gamma > \frac{n}{6}$, then
\begin{eqnarray*}
&&\int_{\mathbb{S}^{n}} K \left(\sum_{i = 1}^p\alpha_i\delta_{a_i, \lambda_i}\right)^{\frac{n + 2\gamma}{n - 2\gamma}}v\,\,dvol_{g_0}\\ && = O\left(\left(\int_{\mathbb{S}^n}vP_\gamma v\right)^{\frac{1}{2}}\right)\left(O\left(\sum_{i\neq j} \left(\varepsilon_{ij}\right)^{\frac{n + 2\gamma}{2n}} \left(\log \varepsilon_{ij}^{-1}\right)^{\frac{n - 2\gamma}{n}}\right)+ \sum_{i = 1}^p\left(\frac{|\nabla K(a_i)|}{\lambda_i}+\frac{1}{\lambda_i^2}\right)\right).
\end{eqnarray*}
\end{lemma}
\begin{proof}
\begin{eqnarray*}
&&\int_{\mathbb{S}^{n}} K \left(\sum_{i = 1}^p\alpha_i\delta_{a_i, \lambda_i}\right)^{\frac{n + 2\gamma}{n - 2\gamma}}v\,\,dvol_{g_0} = \sum_{i = 1}^p\alpha_i^{\frac{n + 2\gamma}{n - 2\gamma}}\int_{\mathbb{S}^{n}} K \delta_{a_i, \lambda_i}^{\frac{n + 2\gamma}{n - 2\gamma}}v\,\,dvol_{g_0} + \\
&& O\left(\left(\int_{\mathbb{S}^n}vP_\gamma v\right)^{\frac{1}{2}}\right)\left[\int_{\mathbb{S}^n}\sum_{i\neq j}(\alpha_i\delta_{a_i, \lambda_i})^{\frac{8n\gamma}{n^2 - 4\gamma^2}}\inf[(\alpha_i\delta_{a_i, \lambda_i})^{\frac{2n}{n + 2\gamma}}, (\alpha_j\delta_{a_j, \lambda_j})^{\frac{2n}{n +2 \gamma}}]\,\,dvol_{g_0}\right]^{\frac{n + 2\gamma}{2n}}.
\end{eqnarray*}
Since
\begin{eqnarray*}
&&\left[\int_{\mathbb{S}^n}\sum_{i\neq j}(\alpha_i\delta_{a_i, \lambda_i})^{\frac{8n\gamma}{n^2 - 4\gamma^2}}\inf[(\alpha_i\delta_{a_i, \lambda_i})^{\frac{2n}{n +2\gamma}}, (\alpha_j\delta_{a_j, \lambda_j})^{\frac{2n}{n + 2\gamma}}]dvol_{g_0}\right]^{\frac{n + 2\gamma}{2n}}\\ &&= O\left(\sum_{i\neq j} \left(\varepsilon_{ij}\right)^{\frac{n + 2\gamma}{2(n - 2\gamma)}} \left(\log \varepsilon_{ij}^{-1}\right)^{\frac{n + 2\gamma}{2n}}\right)
\end{eqnarray*}
when $0 < \gamma \leq \frac{n}{6}$ and
\begin{eqnarray*}
&&\left[\int_{\mathbb{S}^3}\sum_{i\neq j}(\alpha_i\delta_{a_i, \lambda_i})^{\frac{8n\gamma}{n^2 - 4\gamma^2}}\inf[(\alpha_i\delta_{a_i, \lambda_i})^{\frac{2n}{n+2\gamma}}, (\alpha_j\delta_{a_j, \lambda_j})^{\frac{2n}{n + 2\gamma}}]\,\,dvol_{g_0}\right]^{\frac{n + 2\gamma}{2n}}\\ &&= O\left(\sum_{i\neq j} \left(\varepsilon_{ij}\right)^{\frac{n + 2\gamma}{2n}} \left(\log \varepsilon_{ij}^{-1}\right)^{\frac{n - 2\gamma}{n}}\right)
\end{eqnarray*}
when $1 > \gamma > \frac{n}{6}$.
Now, we estimate $\int_{\mathbb{S}^{n}} K \alpha_i^{\frac{n + 2\gamma}{n - 2\gamma}}\delta_{a_i, \lambda_i}^{\frac{n + 2\gamma}{n - 2\gamma}}v\,\,dvol_{g_0}$.
\begin{eqnarray*}
\int_{\mathbb{S}^{n}} K\delta_{a_i, \lambda_i}^{\frac{n + 2\gamma}{n - 2\gamma}}v\,\,dvol_{g_0} &=& \int_{\mathbb{R}^{n}} \left(\tilde{K}(x') - \tilde{K}(g_i)\right)w_{g_i, \lambda_i}^{\frac{n + 2\gamma}{n - 2\gamma}}v\,\,dx',\quad x' = F(x),\,\,g_i = F(a_i)\\
& = &\int_{B(g_i, \rho)} \left(\nabla \tilde{K}(g_i)(x' - g_i) + o(|x' - g_i|)\right)w_{g_i, \lambda_i}^{\frac{n + 2\gamma}{n - 2\gamma}}v\,\,dx'\\
&& + \int_{B^c(g_i, \rho)} \left(\tilde{K}(x') - \tilde{K}(g_i)\right)w_{g_i, \lambda_i}^{\frac{n + 2\gamma}{n - 2\gamma}}v\,\,dx'\\
& = &O\left(\left(\int_{\mathbb{S}^n}vP_\gamma v\right)^{\frac{1}{2}}\right)\left(\frac{|\nabla K(a_i)|}{\lambda_i} + \frac{1}{\lambda_i^2}\right).
\end{eqnarray*}
This completes the proof.
\end{proof}
\begin{lemma}
For any $u = \sum_{i = 1}^p \alpha_i \delta_{a_i, \lambda_i} + v \in V(p, \varepsilon)$, we have
\begin{eqnarray*}
&&\int_{\mathbb{S}^{n}} K(x) \left(\sum_{i = 1}^p\alpha_i\delta_{a_i, \lambda_i}\right)^{\frac{4\gamma}{n - 2\gamma}}v^2dvol_{g_0}\\  &&=  \sum_{i = 1}^p\alpha_i^{\frac{4\gamma}{n - 2\gamma}}K(a_i)\int_{\mathbb{S}^{n}} \delta_{a_i, \lambda_i}^{\frac{4\gamma}{n - 2\gamma}}v^2dvol_{g_0}
\\ &&+O\left(\int_{\mathbb{S}^n}v P_\gamma v\,\,dvol_{g_0}\right)\left(\sum_{i = 1}^p \frac{|\nabla K(a_i)|}{\lambda_i} + \sum_{i = 1}^p\frac{1}{\lambda_i^{2\gamma}} + \sum_{i\neq j}\varepsilon_{ij}^{\frac{2\gamma}{n - 2\gamma}}(\log \varepsilon_{ij}^{-1})^{\frac{2\gamma}{n}}\right).
\end{eqnarray*}
\end{lemma}
\begin{proof}
First, we have
\begin{eqnarray*}
&&\int_{\mathbb{S}^{n}} K(x) \left(\sum_{i = 1}^p\alpha_i\delta_{a_i, \lambda_i}\right)^{\frac{4\gamma}{n - 2\gamma}}v^2dvol_{g_0}\\ &&= \int_{\mathbb{R}^{n}} \tilde{K}(x') \left(\sum_{i = 1}^p\alpha_iw_{g_i, \lambda_i}\right)^{\frac{4\gamma}{n - 2\gamma}}v^2dx'\\
&& = \sum_{i = 1}^p\alpha_i^{\frac{4\gamma}{n - 2\gamma}}\int_{\mathbb{R}^{n}} \tilde{K}(x')w_{g_i, \lambda_i}^{\frac{4\gamma}{n - 2\gamma}}v^2dx'\\ &&+ O\left(\left(\int_{\mathbb{S}^n}v P_\gamma v\,\,dvol_{g_0}\right)\left(\varepsilon_{ij}^{\frac{2\gamma}{n - 2\gamma}}(\log \varepsilon_{ij}^{-1})^{\frac{2\gamma}{n}}\right)\right).
\end{eqnarray*}
Now, we compute $$\sum_{i = 1}^p\alpha_i^{\frac{4\gamma}{n - 2\gamma}}\int_{\mathbb{R}^{n}} \tilde{K}(x')w_{g_i, \lambda_i}^{\frac{4\gamma}{n - 2\gamma}}v^2dx'.$$ Set $\tilde{K}(x') = \tilde{K}(x') - \tilde{K}(g_i) + \tilde{K}(g_i)$, it yields that
\begin{eqnarray*}
&&\sum_{i = 1}^p\alpha_i^{\frac{4\gamma}{n - 2\gamma}}\int_{\mathbb{R}^{n}} \tilde{K}(x')w_{g_i, \lambda_i}^{\frac{4\gamma}{n - 2\gamma}}v^2dx'\\ &&= \sum_{i = 1}^p\alpha_i^{\frac{4\gamma}{n - 2\gamma}}\tilde{K}(g_i)\int_{\mathbb{R}^{n}} w_{g_i, \lambda_i}^{\frac{4\gamma}{n - 2\gamma}}v^2dx'\\
&&+\sum_{i = 1}^p\alpha_i^{\frac{4\gamma}{n - 2\gamma}}\int_{B_i} \left(\tilde{K}(x') - \tilde{K}(g_i)\right)w_{g_i, \lambda_i}^{\frac{4\gamma}{n - 2\gamma}}v^2dx'\\
&&+\sum_{i = 1}^p\alpha_i^{\frac{4\gamma}{n - 2\gamma}}\int_{B_i^c} \left(\tilde{K}(x') - \tilde{K}(g_i)\right)w_{g_i, \lambda_i}^{\frac{4\gamma}{n - 2\gamma}}v^2dx'.
\end{eqnarray*}
Therefore
\begin{eqnarray*}
\sum_{i = 1}^p\alpha_i^{\frac{4\gamma}{n - 2\gamma}}\int_{\mathbb{R}^{n}} \tilde{K}(x')w_{g_i, \lambda_i}^{\frac{4\gamma}{n - 2\gamma}}v^2dx' &=& \sum_{i = 1}^p\alpha_i^{\frac{4\gamma}{n - 2\gamma}}\tilde{K}(g_i)\int_{\mathbb{R}^{n}} w_{g_i, \lambda_i}^{\frac{4\gamma}{n - 2\gamma}}v^2dx'\\
&+&\sum_{i = 1}^p\alpha_i^{\frac{4\gamma}{n - 2\gamma}}\int_{B(0, \rho)} \left(\nabla\tilde{K}(0)\cdot x' + o(|x'|)\right)w_{0, \lambda_i}^{\frac{4\gamma}{n - 2\gamma}}v^2dx'\\
&+&\sum_{i = 1}^p\alpha_i^{\frac{4\gamma}{n - 2\gamma}}\int_{B_i^c} \left(\tilde{K}(x') - \tilde{K}(g_i)\right)w_{g_i, \lambda_i}^{\frac{4\gamma}{n - 2\gamma}}v^2dx'\\
& = & \sum_{i = 1}^p\alpha_i^{\frac{4\gamma}{n - 2\gamma}}\tilde{K}(g_i)\int_{\mathbb{R}^{n}} w_{g_i, \lambda_i}^{\frac{4\gamma}{n - 2\gamma}}v^2dx'\\
&+&O\left(\int_{\mathbb{S}^n}v P_\gamma v\,\,dvol_{g_0}\right)\left(\sum_{i = 1}^p \frac{|\nabla K(a_i)|}{\lambda_i} + \sum_{i = 1}^p\frac{1}{\lambda_i^{2\gamma}}\right).
\end{eqnarray*}
This completes the proof.
\end{proof} 
\section{Proof of Lemma \ref{l:expansionofgradient1}.}
First, we have
$$
J(u) = \lambda (u)\int_{\mathbb{S}^{n}}u P_\gamma u\,\, dvol_{g_0},\quad \lambda(u) = \left(\int_{\mathbb{S}^{n}} K u^{\frac{2n}{n - 2\gamma}}dvol_{g_0}\right)^{-\frac{n - 2\gamma}{n}},
$$
\begin{eqnarray*}
\lambda'(u) W &=& -2\left(\int_{\mathbb{S}^{n}} K u^{\frac{2n}{n - 2\gamma}}dvol_{g_0}\right)^{-\frac{2n-2\gamma}{n}}\left(\int_{\mathbb{S}^{n}} K u^{\frac{n + 2\gamma}{n - 2\gamma}}Wdvol_{g_0}\right)\\
& = & -2\lambda(u)^{\frac{2n - 2\gamma}{n - 2\gamma}}\left(\int_{\mathbb{S}^{n}} K u^{\frac{n + 2\gamma}{n - 2\gamma}}Wdvol_{g_0}\right).
\end{eqnarray*}
Therefore
\begin{eqnarray*}
&&J'(u) W = \lambda' (u)W\int_{\mathbb{S}^{n}}u P_\gamma u\,\, dvol_{g_0} + 2\lambda(u)\int_{\mathbb{S}^{n}}P_\gamma u W\,\, dvol_{g_0}\\
&& =  2\lambda(u)\left[-\lambda(u)^{\frac{n}{n-2\gamma}}\left(\int_{\mathbb{S}^{n}} K u^{\frac{n + 2\gamma}{n - 2\gamma}}Wdvol_{g_0}\right)\int_{\mathbb{S}^{n}}u P_\gamma u\,\, dvol_{g_0} + \int_{\mathbb{S}^{n}}P_\gamma u W\,\, dvol_{g_0}\right]\\
&& =  2\lambda(u)\left[\int_{\mathbb{S}^{n}}P_\gamma u W\,\, dvol_{g_0} -\lambda(u)^{\frac{n}{n-2\gamma}}\left(\int_{\mathbb{S}^{n}} K u^{\frac{n + 2\gamma}{n - 2\gamma}}Wdvol_{g_0}\right)\int_{\mathbb{S}^{n}}u P_\gamma u\,\, dvol_{g_0} \right].
\end{eqnarray*}
Since $u = \sum_{i = 1}^p\alpha_i \delta_{a_i, \lambda_i}\in V(p, \varepsilon)\subset \Sigma^+$, we have $\int_{\mathbb{S}^{n}}u P_\gamma u\,\, dvol_{g_0} = 1$. Thus
\begin{eqnarray*}
J'(u) W = 2\lambda(u)\left[\int_{\mathbb{S}^{n}}P_\gamma u W\,\, dvol_{g_0} -\lambda(u)^{\frac{n}{n-2\gamma}}\left(\int_{\mathbb{S}^{n}} K u^{\frac{n+2\gamma}{n-2\gamma}}Wdvol_{g_0}\right)\right].
\end{eqnarray*}

Let $W = \lambda_j\frac{\partial \delta_{a_j, \lambda_j}}{\partial \lambda_j}$, then we have
\begin{eqnarray*}
&&J'(u) \left(\lambda_j\frac{\partial \delta_{a_j, \lambda_j}}{\partial \lambda_j}\right) =\\ &&2\lambda(u)\left[\langle \sum_{i = 1}^p\alpha_i \delta_{a_i, \lambda_i} ,\lambda_j\frac{\partial \delta_{a_j, \lambda_j}}{\partial \lambda_j}\rangle-\lambda(u)^{\frac{n}{n-2\gamma}}\left(\int_{\mathbb{S}^{n}} K \left(\sum_{i = 1}^p\alpha_i \delta_{a_i, \lambda_i}\right)^{\frac{n + 2\gamma}{n - 2\gamma}}\lambda_j\frac{\partial \delta_{a_j, \lambda_j}}{\partial \lambda_j}dvol_{g_0}\right)\right]
\end{eqnarray*}
and
\begin{eqnarray*}
&&\int_{\mathbb{S}^{n}} K \left(\sum_{i = 1}^p\alpha_i \delta_{a_i, \lambda_i}\right)^{\frac{n + 2\gamma}{n - 2\gamma}}\lambda_j\frac{\partial \delta_{a_j, \lambda_j}}{\partial \lambda_j}dvol_{g_0}= \\ &&\int_{\mathbb{S}^{n}} K \left[\alpha_j^{\frac{n + 2\gamma}{n - 2\gamma}} \delta_{a_j, \lambda_j}^{\frac{n + 2\gamma}{n - 2\gamma}} + \sum_{i\neq j}\alpha_i^{\frac{n + 2\gamma}{n - 2\gamma}} \delta_{a_i, \lambda_i}^{\frac{n + 2\gamma}{n - 2\gamma}} + \frac{n + 2\gamma}{n - 2\gamma}\alpha_j^{\frac{4\gamma}{n-2\gamma}} \delta_{a_j, \lambda_j}^{\frac{4\gamma}{n-2\gamma}}\left(\sum_{i\neq j}\alpha_i \delta_{a_i, \lambda_i}\right)\right]\lambda_j\frac{\partial \delta_{a_j, \lambda_j}}{\partial \lambda_j}dvol_{g_0}\\
&& + \int_{\mathbb{S}^{n}} K \left[\sum_{k\neq j, i\neq j}O\left(\delta_{a_k, \lambda_k}^{\frac{4\gamma}{n-2\gamma}}\delta_{a_i, \lambda_i}\right) + \sum_{k\neq j}O\left(\delta^{\frac{6\gamma - n}{n - 2\gamma}}_{a_j, \lambda_j}\delta^{2}_{a_k, \lambda_k}\right)\right]\lambda_j\frac{\partial \delta_{a_j, \lambda_j}}{\partial \lambda_j}dvol_{g_0}.
\end{eqnarray*}
Moreover, the following estimates hold.
\begin{lemma}
\begin{eqnarray*}
\langle \delta_{a_j,\lambda_j}, \lambda_i\frac{\partial \delta_{a_i, \lambda_i}}{\partial \lambda_i}\rangle = \lambda_i\frac{\partial}{\partial \lambda_i}\left(\int_{\mathbb{S}^3}\delta_{a_j,\lambda_j}^{\frac{n + 2\gamma}{n - 2\gamma}}\delta_{a_i, \lambda_i}dvol_{g_0}\right) = c_0^{\frac{2n}{n - 2\gamma}}c_1\omega_n\lambda_i\frac{\partial\varepsilon_{ij}}{\partial \lambda_i} + o(\varepsilon_{ij}).
\end{eqnarray*}
\end{lemma}
\begin{lemma}
\begin{eqnarray*}
\langle \delta_{a_i,\lambda_i}, \lambda_i\frac{\partial \delta_{a_i, \lambda_i}}{\partial \lambda_i}\rangle = 0.
\end{eqnarray*}
\end{lemma}
\begin{lemma}
\begin{eqnarray*}
&&\int_{\mathbb{S}^n}K(x)\delta_{a_j,\lambda_j}^{\frac{n + 2\gamma}{n - 2\gamma}}\lambda_i\frac{\partial \delta_{a_i, \lambda_i}}{\partial \lambda_i}dvol_{g_0}\\ &&= \int_{\mathbb{R}^n}\tilde{K}(x')w_{g_j,\lambda_j}^{\frac{n + 2\gamma}{n - 2\gamma}}\lambda_i\frac{\partial w_{g_i, \lambda_i}}{\partial \lambda_i}\theta_0\wedge d\theta_0\\
&& = \tilde{K}(g_j)\int_{\mathbb{R}^n}w_{g_i,\lambda_i}^{\frac{n + 2\gamma}{n - 2\gamma}}\lambda_i\frac{\partial w_{g_i, \lambda_i}}{\partial \lambda_i}\theta_0\wedge d\theta_0 +
\int_{B_i}\left(\tilde{K}(x') - \tilde{K}(g_j)\right)w_{g_j,\lambda_j}^{\frac{n + 2\gamma}{n - 2\gamma}}\lambda_i\frac{\partial w_{g_i, \lambda_i}}{\partial \lambda_i}\theta_0\wedge d\theta_0\\
&& +\int_{B_i^c}\left(\tilde{K}(x') - \tilde{K}(g_j)\right)w_{g_j,\lambda_j}^{\frac{n + 2\gamma}{n - 2\gamma}}\lambda_i\frac{\partial w_{g_i, \lambda_i}}{\partial \lambda_i}\theta_0\wedge d\theta_0\\ &&= K(a_j)c_0^{\frac{2n}{n - 2\gamma}}c_1\omega_n\lambda_i\frac{\partial\varepsilon_{ij}}{\partial \lambda_i} + o(\varepsilon_{ij})\\
&& + O\left(\frac{1}{\lambda_i^{\frac{n - 2\gamma}{2}}\lambda_j^{\frac{n+2\gamma}{2}}}\right) + O(\varepsilon_{ij}^{\frac{n}{n-2\gamma}}\log \varepsilon_{ij}^{-1}).
\end{eqnarray*}
\end{lemma}
\begin{lemma}
\begin{eqnarray*}
\int_{\mathbb{S}^n}K(x)\delta_{a_i,\lambda_i}^{\frac{n + 2\gamma}{n - 2\gamma}}\lambda_i\frac{\partial \delta_{a_i, \lambda_i}}{\partial \lambda_i}dvol_{g_0} &=& \int_{\mathbb{R}^n}\tilde{K}(x')w_{g_i,\lambda_i}^{\frac{n + 2\gamma}{n - 2\gamma}}\lambda_i\frac{\partial w_{g_i, \lambda_i}}{\partial \lambda_i}\theta_0\wedge d\theta_0\\
& = &\Delta\tilde{K}(g_i)\int_{B_i}|x' - g_i|^2w_{g_i,\lambda_i}^{\frac{n + 2\gamma}{n - 2\gamma}}\lambda_i\frac{\partial w_{g_i, \lambda_i}}{\partial \lambda_i}\theta_0\wedge d\theta_0\\
&&+O\left(\int_{B_i}|x' - g_i|^2w_{g_i,\lambda_i}^{\frac{n + 2\gamma}{n - 2\gamma}}\lambda_i\frac{\partial w_{g_i, \lambda_i}}{\partial \lambda_i}\theta_0\wedge d\theta_0\right)\\
&&+\int_{B_i^c}\left(\tilde{K}(x') - \tilde{K}(g_i)\right)w_{g_i,\lambda_i}^{\frac{n + 2\gamma}{n - 2\gamma}}\lambda_i\frac{\partial w_{g_i, \lambda_i}}{\partial \lambda_i}\theta_0\wedge d\theta_0\\
&=& - \frac{n - 2\gamma}{2n}\frac{\Delta K(a_i)}{\lambda_i^2}c_2(1 + o(1)).
\end{eqnarray*}
\end{lemma}
\begin{lemma}For $i\neq k$,
\begin{eqnarray*}
&&\int_{\mathbb{S}^n}K(x)\delta_{a_i,\lambda_i}^{\frac{4\gamma}{n-2\gamma}}\lambda_i\frac{\partial \delta_{a_i, \lambda_i}}{\partial \lambda_i}\delta_{a_k,\lambda_k}dvol_{g_0}\\ &&= \frac{n - 2\gamma}{n + 2\gamma}c_0^{\frac{2n}{n - 2\gamma}}K(a_i)c_1\omega_n\lambda_i\frac{\partial\varepsilon_{ik}}{\partial \lambda_i} + o(\varepsilon_{ik})+O\left(\frac{1}{\lambda_i^{\frac{n +2 \gamma}{2}}\lambda_k^{\frac{n - 2\gamma}{2}}}\right).
\end{eqnarray*}
\end{lemma}
\begin{lemma}
For $i\neq k$, it holds that
\begin{eqnarray*}
&&\int_{\mathbb{S}^n}K(x)\delta^{\frac{6\gamma - n}{n - 2\gamma}}_{a_i, \lambda_i}\delta^{2}_{a_k, \lambda_k}\lambda_i\frac{\partial \delta_{a_i, \lambda_i}}{\partial \lambda_i}\delta_{a_i,\lambda_i}dvol_{g_0} = O\left(\varepsilon_{ik}^{\frac{n}{n-2\gamma}}\log \varepsilon_{ik}^{-1}\right).
\end{eqnarray*}
\end{lemma}
\begin{lemma}
For $j\neq k$, $k\neq i$ and $i\neq j$, it holds that
\begin{eqnarray*}
&&\int_{\mathbb{S}^n}K(x)\delta_{a_k,\lambda_k}^{\frac{4\gamma}{n-2\gamma}}\lambda_i\frac{\partial \delta_{a_i, \lambda_i}}{\partial \lambda_i}\delta_{a_j,\lambda_j}dvol_{g_0} = O\left(\varepsilon_{ik}^{\frac{n}{n-2\gamma}}\log \varepsilon_{ik}^{-1}\right) + O\left(\varepsilon_{jk}^{\frac{n}{n-2\gamma}}\log \varepsilon_{jk}^{-1}\right).
\end{eqnarray*}
\end{lemma}
Using the estimates above, we have
\begin{eqnarray*}
&&J'(u) \left(\lambda_j\frac{\partial \delta_{a_j, \lambda_j}}{\partial \lambda_j}\right) =\\ &&2\lambda(u)\left[\langle \sum_{i = 1}^p\alpha_i \delta_{a_i, \lambda_i} ,\lambda_j\frac{\partial \delta_{a_j, \lambda_j}}{\partial \lambda_j}\rangle-\lambda(u)^{\frac{n}{n-2\gamma}}\left(\int_{\mathbb{S}^{n}} K \left(\sum_{i = 1}^p\alpha_i \delta_{a_i, \lambda_i}\right)^{\frac{n + 2\gamma}{n - 2\gamma}}\lambda_j\frac{\partial \delta_{a_j, \lambda_j}}{\partial \lambda_j}dvol_{g_0}\right)\right]\\
&& = 2\lambda(u)\left[\sum_{i\neq j}\alpha_i c_0^{\frac{2n}{n - 2\gamma}}c_1\omega_n\lambda_j\frac{\partial\varepsilon_{ij}}{\partial \lambda_j}(1 + o(1)) + o(\sum_{i\neq j}\varepsilon_{ij})\right]
- 2\lambda(u)^{\frac{2n - 2\gamma}{n - 2\gamma}}\\ &&\left[- \alpha_j^{\frac{n + 2\gamma}{n - 2\gamma}}\frac{n - 2\gamma}{2n}c_2\frac{\Delta K(a_j)}{\lambda_j^2} (1 + o(1))+ \sum_{i\neq j}\alpha_i^{\frac{n + 2\gamma}{n - 2\gamma}}c_0^{\frac{2n}{n - 2\gamma}}c_1\omega_n\lambda_j\frac{\partial\varepsilon_{ij}}{\partial \lambda_j}K(a_i)(1 + o(1)) \right]\\
&&- 2\lambda(u)^{\frac{2n - 2\gamma}{n - 2\gamma}} \left[\sum_{i\neq j}\alpha_i\alpha_j^{\frac{4\gamma}{n-2\gamma}}\frac{n - 2\gamma}{n + 2\gamma}c_0^{\frac{2n}{n - 2\gamma}}c_1\omega_n\lambda_j\frac{\partial\varepsilon_{ij}}{\partial \lambda_j}K(a_j)(1 + o(1))\right.
\end{eqnarray*}
\begin{eqnarray*}
\left. + o(\sum_{i\neq j}\varepsilon_{ij}) + O\left(\sum_{i\neq j}\varepsilon_{ij}^{\frac{n}{n-2\gamma}}\log \varepsilon_{ij}^{-1}\right)\right].
\end{eqnarray*}
Since $\lambda(u)^{\frac{n}{n - 2\gamma}}\alpha_j^{\frac{4\gamma}{n-2\gamma}}K(a_j)\to 1$, we have the desired estimate of Lemma \ref{l:expansionofgradient1}.

\section{Proof of Lemma \ref{l:expansionofgradient2}.}
Let $W = \frac{1}{\lambda_j}\frac{\partial \delta_{a_j, \lambda_j}}{\partial a_j}$, then we have
\begin{eqnarray*}
&& J'(u) \left(\frac{1}{\lambda_j}\frac{\partial \delta_{a_j, \lambda_j}}{\partial a_j}\right) =\\
&&2\lambda(u)\left[\langle \sum_{i = 1}^p\alpha_i \delta_{a_i, \lambda_i} ,\frac{1}{\lambda_j}\frac{\partial \delta_{a_j, \lambda_j}}{\partial a_j}\rangle-\lambda(u)^{\frac{n}{n - 2\gamma}}\left(\int_{\mathbb{S}^{n}} K \left(\sum_{i = 1}^p\alpha_i \delta_{a_i, \lambda_i}\right)^{\frac{n + 2\gamma}{2-2\gamma}}\frac{1}{\lambda_j}\frac{\partial \delta_{a_j, \lambda_j}}{\partial a_j}\right)\right].
\end{eqnarray*}
and
\begin{eqnarray*}
&&\int_{\mathbb{S}^{n}} K \left(\sum_{i = 1}^p\alpha_i \delta_{a_i, \lambda_i}\right)^{\frac{n+2\gamma}{n-2\gamma}}\frac{1}{\lambda_j}\frac{\partial \delta_{a_j, \lambda_j}}{\partial a_j}dvol_{g_0}=\\ && \int_{\mathbb{S}^{n}} K \left[\alpha_j^{\frac{n+2\gamma}{n-2\gamma}} \delta_{a_j, \lambda_j}^{\frac{n+2\gamma}{n-2\gamma}} + \sum_{i\neq j}\alpha_i^{\frac{n+2\gamma}{n-2\gamma}} \delta_{a_i, \lambda_i}^{\frac{n+2\gamma}{n-2\gamma}} + \frac{n+2\gamma}{n-2\gamma}\alpha_j^{\frac{4\gamma}{n-2\gamma}} \delta_{a_j, \lambda_j}^{\frac{4\gamma}{n-2\gamma}}\left(\sum_{i\neq j}\alpha_i \delta_{a_i, \lambda_i}\right)\right]\frac{1}{\lambda_j}\frac{\partial \delta_{a_j, \lambda_j}}{\partial a_j}dvol_{g_0}\\
&& + \int_{\mathbb{S}^{n}} K \left[\sum_{k\neq j, i\neq j}O\left(\delta_{a_k, \lambda_k}^{\frac{4\gamma}{n-2\gamma}}\delta_{a_i, \lambda_i}\right) + \sum_{k\neq j}O\left(\delta^{\frac{6\gamma - n}{n - 2\gamma}}_{a_j, \lambda_j}\delta^{2}_{a_k, \lambda_k}\right)\right]\frac{1}{\lambda_j}\frac{\partial \delta_{a_j, \lambda_j}}{\partial a_j}dvol_{g_0}.
\end{eqnarray*}
Moreover, we have the following estimates.
\begin{lemma}
\begin{eqnarray*}
\langle \delta_{a_j,\lambda_j}, \frac{1}{\lambda_j}\frac{\partial \delta_{a_j, \lambda_j}}{\partial a_j}\rangle = 0.
\end{eqnarray*}
\end{lemma}
\begin{lemma}
\begin{eqnarray*}
\langle \delta_{a_i,\lambda_i}, \frac{1}{\lambda_j}\frac{\partial \delta_{a_j, \lambda_j}}{\partial a_j}\rangle = c_0^{\frac{2n}{n-2\gamma}}c_1\omega_n\frac{1}{\lambda_j}\frac{\partial \varepsilon_{ij}}{\partial a_j}+ \frac{1}{\lambda_j}O(\varepsilon_{ij}^{\frac{n + 1}{n-2\gamma}}\lambda_i\lambda_j d(a_i, a_j)).
\end{eqnarray*}
\end{lemma}
\begin{lemma}
Since $w_{g_j,\lambda_j}^{\frac{n + 2\gamma}{n - 2\gamma}}\frac{1}{\lambda_j}\frac{\partial w_{g_j, \lambda_j}}{\partial g_j} = \frac{n -2 \gamma}{2n}\frac{1}{\lambda_j}\frac{\partial}{\partial g_j}\left(w_{g_j, \lambda_j}\right)^{\frac{2n}{n - 2\gamma}} = (n - 2\gamma)|x' - g_j|w_{g_j, \lambda_j}^{\frac{2n + 2}{n - 2\gamma}}$, it holds that
\begin{eqnarray*}
\int_{\mathbb{S}^n}K(x)\delta_{a_j,\lambda_j}^{\frac{n + 2\gamma}{n - 2\gamma}}\frac{1}{\lambda_j}\frac{\partial \delta_{a_j, \lambda_j}}{\partial a_j}dvol_{g_0} &=& \int_{\mathbb{R}^n}\tilde{K}(x')w_{g_j,\lambda_j}^{\frac{n + 2\gamma}{n - 2\gamma}}\frac{1}{\lambda_j}\frac{\partial w_{g_j, \lambda_j}}{\partial g_j}dx'\\
& = &\int_{B_j}\left(\tilde{K}(x') - \tilde{K}(g_j)\right)w_{g_j,\lambda_j}^{\frac{n + 2\gamma}{n - 2\gamma}}\frac{1}{\lambda_j}\frac{\partial w_{g_j, \lambda_j}}{\partial a_j}dx'\\
&&+\int_{B_j^c}\left(\tilde{K}(x') - \tilde{K}(g_j)\right)w_{g_j,\lambda_j}^{\frac{n + 2\gamma}{n - 2\gamma}}\frac{1}{\lambda_j}\frac{\partial w_{g_j, \lambda_j}}{\partial a_j}dx'\\
& = &\int_{B_j}w_{g_j,\lambda_j}^{\frac{n + 2\gamma}{n - 2\gamma}}\frac{1}{\lambda_j}\frac{\partial w_{g_j, \lambda_j}}{\partial a_j}\nabla \tilde{K}(g_j)(x' - g_j)dx'\\
&&+ O\left(\sup\left|\nabla^2 \tilde{K}(g_j)\right|\int_{B_j}|x' - g_j|^3w_{g_j, \lambda_j}^{\frac{2n + 2}{n - 2\gamma}}dx'\right)\\
&&+\int_{B_j^c}\left(\tilde{K}(x') - \tilde{K}(g_j)\right)w_{g_j,\lambda_j}^{\frac{n + 2\gamma}{n - 2\gamma}}\frac{1}{\lambda_j}\frac{\partial w_{g_j, \lambda_j}}{\partial a_j}dx'\\
&=& \frac{n - 2\gamma}{2n}c_0^{\frac{2n}{n-2\gamma}}c_1\omega_n\frac{\nabla K(a_j)}{\lambda_j} + O(\frac{1}{\lambda_j^3}).
\end{eqnarray*}
\end{lemma}
\begin{lemma}
For $i\neq j$, it holds that
\begin{eqnarray*}
&&\int_{\mathbb{S}^n}K(x)\delta_{a_j,\lambda_j}^{\frac{n + 2\gamma}{n - 2\gamma}}\frac{1}{\lambda_i}\frac{\partial \delta_{a_i, \lambda_i}}{\partial a_i}dvol_{g_0}\\ &&= \int_{\mathbb{R}^n}\tilde{K}(x')w_{g_j,\lambda_j}^{\frac{n + 2\gamma}{n - 2\gamma}}\frac{1}{\lambda_i}\frac{\partial w_{g_i, \lambda_i}}{\partial a_i}dx'\\
&& = \tilde{K}(g_j)\int_{\mathbb{R}^n}w_{g_j,\lambda_j}^{\frac{n + 2\gamma}{n - 2\gamma}}\frac{1}{\lambda_i}\frac{\partial w_{g_i, \lambda_i}}{\partial a_i}dx' +
\int_{B_i}\left(\tilde{K}(x') - \tilde{K}(g_j)\right)w_{g_j,\lambda_j}^{\frac{n + 2\gamma}{n - 2\gamma}}\frac{1}{\lambda_i}\frac{\partial w_{g_i, \lambda_i}}{\partial a_i}dx'\\
&& +\int_{B_i^c}\left(\tilde{K}(x') - \tilde{K}(g_j)\right)w_{g_j,\lambda_j}^{\frac{n + 2\gamma}{n - 2\gamma}}\frac{1}{\lambda_i}\frac{\partial w_{g_i, \lambda_i}}{\partial a_i}dx'\\ &&= c_0^{\frac{2n}{n-2\gamma}}c_1\omega_n\frac{K(a_j)}{\lambda_j}\frac{\partial \varepsilon_{ij}}{\partial a_j} + o(\varepsilon_{ij})+ O\left(\frac{1}{\lambda_j^{\frac{n + 2\gamma}{2}}\lambda_i^{\frac{n - 2\gamma}{2}}}\right).
\end{eqnarray*}
\end{lemma}
\begin{lemma}
For $i\neq k$, we have
\begin{eqnarray*}
&&\int_{\mathbb{S}^n}K(x)\delta_{a_i,\lambda_i}^{\frac{4\gamma}{n - 2\gamma}}\frac{1}{\lambda_i}\frac{\partial \delta_{a_i, \lambda_i}}{\partial a_i}\delta_{a_k,\lambda_k}dvol_{g_0} = O\left(\varepsilon_{ik}\right).
\end{eqnarray*}
\end{lemma}
\begin{lemma}
For $i\neq k$, we have
\begin{eqnarray*}
&&\int_{\mathbb{S}^n}K(x)\delta_{a_i,\lambda_i}^{\frac{6\gamma - n}{n - 2\gamma}}\frac{1}{\lambda_i}\frac{\partial \delta_{a_i, \lambda_i}}{\partial a_i}\delta^2_{a_k,\lambda_k}dvol_{g_0} = O\left(\varepsilon_{ik}^{\frac{n}{n - 2\gamma}}\log \varepsilon_{ik}^{-1}\right).
\end{eqnarray*}
\end{lemma}
\begin{lemma}
For $j\neq k$, $k\neq i$ and $i\neq j$, it holds that
\begin{eqnarray*}
&&\int_{\mathbb{S}^n}K(x)\delta_{a_k,\lambda_k}^{\frac{4\gamma}{n - 2\gamma}}\frac{1}{\lambda_i}\frac{\partial \delta_{a_i, \lambda_i}}{\partial a_i}\delta_{a_j,\lambda_j}dvol_{g_0} = O\left(\varepsilon_{ik}^{\frac{n}{n-2\gamma}}\log \varepsilon_{ik}^{-1}\right) + O\left(\varepsilon_{jk}^{\frac{n}{n-2\gamma}}\log \varepsilon_{jk}^{-1}\right).
\end{eqnarray*}
\end{lemma}
Using the lemmas above, we have
\begin{eqnarray*}
&&J'(u) \left(\frac{1}{\lambda_j}\frac{\partial \delta_{a_j, \lambda_j}}{\partial a_j}\right) =\\
&&2\lambda(u)\left[\langle \sum_{i = 1}^p\alpha_i \delta_{a_i, \lambda_i} ,\frac{1}{\lambda_j}\frac{\partial \delta_{a_j, \lambda_j}}{\partial a_j}\rangle-\lambda(u)^{\frac{n}{n - 2\gamma}}\left(\int_{\mathbb{S}^{n}} K \left(\sum_{i = 1}^p\alpha_i \delta_{a_i, \lambda_i}\right)^{\frac{n + 2\gamma}{n - 2\gamma}}\frac{1}{\lambda_j}\frac{\partial \delta_{a_j, \lambda_j}}{\partial a_j}\right)\right]\\
&& = 2\lambda(u)\left[\sum_{i\neq j}\alpha_i c_0^{\frac{2n}{n-2\gamma}}c_1\omega_n\frac{1}{\lambda_j}\frac{\partial \varepsilon_{ij}}{\partial a_j}(1 + o(1)) \right]\\
&&- 2\lambda(u)^{\frac{2n - 2\gamma}{n - 2\gamma}}\left[\alpha_j^{\frac{n + 2\gamma}{n - 2\gamma}}\frac{n - 2\gamma}{2n}c_0^{\frac{2n}{n-2\gamma}}c_1\omega_n\frac{\nabla K(a_j)}{\lambda_j}(1 + o(1))+ O(\sum_{i = 1}^p\frac{1}{\lambda_j^3})\right]\\ &&- 2\lambda(u)^{\frac{2n - 2\gamma}{n - 2\gamma}} \left[\sum_{i\neq j}\alpha_i^{\frac{n+2\gamma}{n-2\gamma}}K(a_i)c_0^{\frac{2n}{n-2\gamma}}c_1\omega_n\frac{1}{\lambda_j}\frac{\partial \varepsilon_{ij}}{\partial a_j}(1 + o(1)) \right]\\
&&- 2\lambda(u)^{\frac{2n - 2\gamma}{n - 2\gamma}} \left[O(\sum_{i\neq j}\varepsilon_{ij}) + O\left(\sum_{i\neq j}\varepsilon_{ij}^{\frac{n}{n-2\gamma}}\log \varepsilon_{ij}^{-1}\right)\right].
\end{eqnarray*}
Since $\lambda(u)^{\frac{n}{n - 2\gamma}}\alpha_j^{\frac{4\gamma}{n-2\gamma}}K(a_j)\to 1$, we have the desired estimate of Lemma \ref{l:expansionofgradient2}. 
\end{appendix}
\newcommand{\Toappear}{to appear in}
\bibliography{mrabbrev,mlabbr2003-0,localbib}
\bibliographystyle{plain}
\end{document}